\documentclass[12pt]{amsart}
\usepackage{amsmath}
\usepackage{amsfonts}
\usepackage{amssymb}
\usepackage{graphicx}
\usepackage{bm}

\newcommand\eps{\varepsilon}
\newcommand\R{{\mathbf{R}}}

\newcommand\Z{{\mathbf{Z}}}

\renewcommand\S{{\mathcal{S}}}
\newcommand\E{{\mathcal{E}}}
\newcommand\M{{\mathcal{M}}}
\renewcommand\L{{\mathcal{L}}}

\theoremstyle{plain}
  \newtheorem{theorem}[subsection]{Theorem}
  \newtheorem{conjecture}[subsection]{Conjecture}
  \newtheorem{proposition}[subsection]{Proposition}
  \newtheorem{lemma}[subsection]{Lemma}
  \newtheorem{corollary}[subsection]{Corollary}

\theoremstyle{remark}
  \newtheorem{remark}[subsection]{Remark}
  
  \newtheorem{example}[subsection]{Example}

\theoremstyle{definition}

\include{psfig}

\begin{document}

\title[sharp linear and bilinear restriction estimates]{Sharp Linear and Bilinear restriction
estimates for paraboloids in the cylindrically symmetric case}
\author{Shuanglin Shao}
\address{Department of Mathematics, UCLA, Los Angeles CA 90095-1555}
\email{slshao@math.ucla.edu}


\vspace{-0.1in}
\begin{abstract}
For cylindrically symmetric functions dyadically supported on the paraboloid, we obtain a family
of sharp linear and bilinear adjoint restriction estimates. As corollaries, we first extend the
ranges of exponents for the classical \textit{linear or bilinear adjoint restriction conjectures}
for such functions and verify the \textit{linear adjoint restriction conjecture} for the
paraboloid. We also interpret the restriction estimates in terms of solutions to the Schr\"odinger
equation and establish the analogous results when the paraboloid is replaced by the lower third of
the sphere.
\end{abstract}

\maketitle

\section{Introduction}
Let $n\ge 3$ be a fixed integer and $S$ be a smooth compact non-empty subset of the paraboloid
$\{(\tau,\xi)\in \R\times \R^{n-1}:\,\tau=|\xi|^2\}$. If $0<p,q\le \infty$, the classical
\emph{linear adjoint restriction estimate} \footnote{In the notation of
\cite{Tao:2008:recent-progress-restric}, the estimate \eqref{eq:lin-restr} is denoted by
$R^{*}_S(p\to q)$ and the estimate \eqref{eq:bilin-restr} is denoted by $R^{*}_{S_1, S_2}(p\times
p\to q)$.} for the paraboloid is the \textit{a priori} estimate
\begin{equation}\label{eq:lin-restr}
\|(gd\sigma)^{\vee} \|_{L_{t,x}^q(\R\times \R^{n-1})}\le C_{p,q,n, S}\|g\|_{L^p(S,d\sigma)}
\end{equation} for all Schwartz functions $g$ on $S$, where
$$(gd\sigma)^{\vee}(t,x)=\int_{S} g(\tau, \xi)e^{i(x\cdot\xi+t\tau)}d\sigma(\xi)
=\int_{\R^{n-1}}g(|\xi|^2, \xi)e^{i(x\cdot\xi+t|\xi|^2)}d\xi$$ denotes the inverse space-time
Fourier transform of the measure $gd\sigma$, and $d\sigma$ is the canonical measure of the
paraboloid defined in Section \ref{sec:notations-statements}. By duality, the estimate
\eqref{eq:lin-restr} is equivalent to the following estimate
\begin{equation*}
\|\hat{f}\|_{L^{p'}(S, d\sigma)}\le C_{p,q,n,S} \|f\|_{L^{q'}(\R\times \R^{n-1})}
\end{equation*} for all Schwartz functions $f$, which roughly says that the Fourier transform
of an $L^{q'}(\R\times \R^{n-1})$ function can be ``meaningfully" restricted to the paraboloid
$S$. This leads to the \textit{restriction problem}, one of the central problems in harmonic
analysis, which concerns the optimal range of exponents $p$ and $q$ for which the estimate
\eqref{eq:lin-restr} should hold. It was originally proposed by Stein for the sphere
\cite{Stein:1979:problems-in-harmonic} and then extended to smooth sub-manifolds of $\R\times
\R^{n-1}$ with appropriate curvature \cite{Stein:1993} such as the paraboloid and the cone. The
restriction problem is intricately related to other outstanding problems in analysis such as the
Bochner-Riesz conjecture, the local smoothing conjecture, the Kakeya set conjecture and the Kakeya
maximal function conjecture, see e.g., \cite{Tao:2008:recent-progress-restric},
\cite{Tao:1999:Bochner-Resiez-restri}.

In this paper, we will mainly focus on the restriction estimates for the paraboloid. The
corresponding \emph{linear adjoint restriction conjecture} for the paraboloid asserts that
\begin{conjecture}\label{con:lin-restr}
The inequality \eqref{eq:lin-restr} holds with constants depending on $S$, $n$ and $p,q$ if and
only if $q>\frac {2n}{n-1}$ and $\frac {n+1}{q}\le \frac {n-1}{p'}$.
\end{conjecture}
The conditions on $p$ and $q$ are known to be best possible by the decay estimates of
$(d\sigma)^{\vee}$ and the standard Knapp example, see e.g., \cite{Stein:1993},
\cite{Tao:2008:recent-progress-restric}. When $n=2$, the non-endpoint case was first proven to be
true by Fefferman and Stein \cite{Fefferman-Stein:1971:maxi-inequalities} (and generalized to
other oscillatory integrals by Carleson and Sj\"olin
\cite{Carleson-Sjolin:1972:oscill-multipl-disc}), and the endpoint case was proven to be true by
Zygmund \cite{Zygmund:1974}. When $n>2$, it was proven with the additional condition $q>\frac
{2(n+1)}{n-1}$ by Tomas \cite{Tomas:1975:restrict} using real interpolation, and $q=\frac
{2(n+1)}{n-1}$ by Stein \cite{Stein:1993} using complex interpolation. In 1977, C\'ordoba
\cite{Cordoba:1977:Kakeya-spher-sum} gave an alternate proof for $n=2$ by largely relying on the
successful resolution of the Kakeya conjecture in two dimensions. In 1991, Bourgain
\cite{Bourgain:1991:Besicovitch-maximal} generalized C\'ordoba's arguments to higher dimensions,
so that nontrivial progress on the Kakeya problem might imply some nontrivial progress on the
restriction result; using this technique, he proved estimates for some $q<\frac {2(n+1)}{n-1}$; in
particular, $q>4-\frac 2{15}$ when $n=3$. Further improvements along this line were made by Moyua,
Vargas, Vega and Wolff, see e.g., \cite{Moyua-Varg-Vega:1996:schrod-maxi},
\cite{Wolff:1995:omproved-bound-kakeya}. The current best result $q>\frac {2(n+2)}{n}$ in higher
dimensions $n\ge 3$ is due to Tao \cite{Tao:2003:paraboloid-restri}, based on the techniques in
Wolff's breakthrough paper on the cone restriction estimates \cite{Wolff:2001:restric-cone}.

Among various techniques developed to attack this problem, the bilinear method proves to be very
powerful. Variants of this idea also have applications to the nonlinear dispersive equations, see
e.g., \cite{Bourgain:1993:Schrodinger-lattice}, \cite{Bourgain:1993:KDV-lattice},
\cite{Klainerman-machedon:1993:spa-time-null-form}, etc. More precisely, we assume $S_1$ and $S_2$
to be two smooth compact non-empty subsets of the paraboloid in $\R\times \R^{n-1}$, which are
transverse in the sense that the unit normals of $S_1$ and of $S_2$ are separated by at least some
fixed angle $c>0$. Then the classical \emph{bilinear adjoint restriction conjecture} concerns the
optimal range of exponents $p$ and $q$ for which the bilinear operator, $(f,g)\to
(fd\sigma_1)^{\vee} (gd\sigma_2)^{\vee}$, should bound from $L^p\times L^p$ to $L^q$, where
$d\sigma_1, d\sigma_2$ are the canonical Lebesgue measures of $S_1$, $S_2$, respectively. The
following formulation of this conjecture is taken from
\cite{Tao-Vargas-Vega:1998:bilinear-restri-kakeya}.
\begin{conjecture}\label{con:bilin-restr}
Let $S_1$, $S_2$ be defined as above and $q\ge \frac {n}{n-1},\frac {n+2}{2q}+\frac {n}{p}\le n$
and $\frac {n+2}{2q}+\frac {n-2}{p}\le n-1$. Then there exists a constant $0<C<\infty$ depending
on $S_1, S_2, n$ and $p,q$ such that
\begin{equation}\label{eq:bilin-restr}
\|(fd\sigma_1)^{\vee} (gd\sigma_2)^{\vee}\|_{L^q_{t,x}(\R\times\R^{n-1})}\le
C\|f\|_{L^p(S_1)}\|g\|_{L^p(S_2)}
\end{equation}
for all $f\in L^p(S_1)$ and $g\in L^p(S_2)$.
\end{conjecture}
It is known that the bilinear restriction conjecture \ref{con:bilin-restr} is stronger than the
linear restriction conjecture \ref{con:lin-restr}, see
\cite{Tao-Vargas-Vega:1998:bilinear-restri-kakeya}. For a discussion of recent progress made on
this problem, see \cite{Tao:2008:recent-progress-restric}. We remark that the conditions on $p$
and $q$ in this conjecture are also known to be best possible by the decay estimates of
$(d\sigma)^{\vee}$ and the variants of the standard Knapp examples such as the squashed caps and
the stretched caps, see e.g., \cite{Tao-Vargas-Vega:1998:bilinear-restri-kakeya},
\cite{Tao:2008:recent-progress-restric}.

However, none of these Knapp-type examples are cylindrically symmetric functions, i.e., functions
on $\R\times\R^{n-1}$ invariant under spatial rotations. Hence we expect that further estimates
are available if we assume that functions are cylindrically symmetric and supported on a dyadic
subset of the paraboloid in the form of $\{(\tau, \xi): M\le |\xi|\le 2M, \tau=|\xi|^2\}$ with
$M\in 2^{\Z}$. We denote by $\L_M$ this class of functions. Indeed, it is the case: when $n=3$,
the Tomas-Stein restriction estimate $L^2\to L^4$ is known to be best possible; but for functions
in $\L_M$, the estimate $L^2\to L^q$ is true for any $q>10/3$ by Corollary \ref{cor:lin-restr} in
Section \ref{sec:linear}.

Our main theorems, Theorem \ref{thm:dyadic-lin} and \ref{thm:dyadic-bilin}, of this paper are to
present a family of sharp linear adjoint restriction estimates for $f\in \L_1$, and bilinear ones
for $f\in \L_1$ and $g\in \L_M$ with $0<M\le 1/4$ on the dyadic space-time slab $\R\times \{R/2\le
|x|\le R\}$ with $R\in 2^{\Z}$. The proofs essentially combine the two classical and elementary
methods, the Carleson-Sj\"olin argument \cite{Stein:1993} and the bilinear method via the Whitney
decomposition, which effectively solved the two dimensional \emph{restriction conjecture}. In the
arguments, we heavily exploit the rotational symmetry via the ``Fourier-Bessel" formula, Lemma
\ref{lem:Four-Bess}, for cylindrically symmetric functions to reduce matters to main term
estimates by encoding the error term into certain integrals. A lot of effort is devoted to
inventing counterexamples to show that the restriction estimates are best possible by relying on
the idea coming from the standard Knapp examples, the principles of both stationary phase and
non-stationary phase \cite{Stein:1993} and the Khintchine inequality
\cite{Tao:2008:recent-progress-restric}. We remark that some of of them are quite challenging, see
e.g., Example \ref{ex:II-larger}.

As corollaries of the main theorems, we can verify the inequality \eqref{eq:lin-restr} for $\L_M$
when the exponents $p$ and $q$ are in a larger region (see Figure \ref{fig:lin}) and show that it
is nearly sharp except for certain endpoints. Furthermore, we show that the \textit{linear adjoint
restriction conjecture} \ref{con:lin-restr} holds for all cylindrically symmetric functions when
$p$ and $q$ are restricted to the classical region. By similar arguments, one can also establish
the analogous sharp restriction estimates when the paraboloid is replaced by the lower third of
the sphere $\S^{n-1}$ or more general cylindrically symmetric and compact hypersurfaces of
elliptic type as defined in \cite{Moyua-Vargas-Vega:1999},
\cite{Tao-Vargas-Vega:1998:bilinear-restri-kakeya}. As applications of the restriction estimates,
we will interpret them in terms of the solutions to the Sch\"odinger equations and present another
proof of the weighted Strichartz estimates in \cite{Vilela:2001:radial-schrod} for Schr\"odinger
equations.

$\textbf{Acknowledgements.}$ The author is very grateful to his advisor Terence Tao for
introducing this fascinating subject, and is indebted to him for many helpful conversations and
encouragement during the preparation of this paper. The author thanks Monica Visan for helpful
discussions. The author would like to thank the referee for his valuable comments and suggestions.

\section{Notations and Main Theorems}\label{sec:notations-statements}
Let $n\ge 3$ be the fixed space-time dimension. In this paper, we interpret $\R\times\R^{n-1}$ as
the space-time frequency space.

We will use the notations $X\lesssim Y$, $Y\gtrsim X$, or $X=O(Y)$ to denote the estimate $|X|\le
C Y$ for some constant $0<C<\infty$, which may depend on $p,q,n$ and $S_1$ or $S_2$, but not on
the functions. If $X\lesssim Y$ and $Y\lesssim X$ we will write $X\sim Y$. If the constant $C$
depends on a special parameter other than the above, we shall denote it explicitly by subscripts.
For example, $C_{\eps}$ should be understood as a positive constant not only depending on $p,q,n$
and $S_1$ or $S_2$, but also on $\eps$.

We denote by $d\sigma$ the canonical Lebesgue measure of the standard paraboloid $S=\{(\tau,\xi):
\tau=|\xi|^2\}$ in $\R\times \R^{n-1}$, which is the pullback of $n-1$-dimensional Lebesgue
measure $d{\xi}$ under the projection map $(\tau, \xi)\mapsto {\xi}$; thus,
\begin{equation*}
\int_{S} f(\tau, \xi) d\sigma =\int_{\R^{n-1}} f(|{\xi}|^2, {\xi})d{\xi} .
\end{equation*}
By $\S^{n-2}$ we denote the ${n-1}$ dimensional unit sphere canonically embedded in $\R^{n-1}$,
and by $d\mu$ its surface measure.

We define a dyadic number to be any number $R\in 2^{\Z}$ of the form $R=2^j$ where $j$ is an
integer. For each dyadic number $R>0$, we define the dyadic annulus in $\R^{n-1}$,
$$A_R:=\{x\in \R^{n-1}: R/2\le |x|\le R\}.$$

We define the space-time norm $L^q_tL^r_x$ of $f$ on $\R\times \R^{n-1}$ by
$$\|f\|_{L^q_tL^r_x(\R\times\R^{n-1})}=\left(\int_{\R}\left(\int_{\R^{n-1}}
|f(t,x)|^{r}d\,x\right)^{q/r}d\,t\right)^{1/q},$$ with the usual modifications when $q$ or $r$ are
equal to infinity, or when the domain $\R\times \R^{n-1}$ is replaced by a small region of
space-time such as $\R\times A_R$. When $q=r$, we abbreviate it by $L^{q}_{t, x}$. Unless
specified in this paper, all the space-time integrals are taken over $\R\times A_R$ with dyadic
$R>0$, which will be clear from the context.

We define the spatial Fourier transform of $f$ on $\R^{n-1}$ by
$$\hat{f}(\xi)=\int_{\R^{n-1}}f(x)e^{-ix\cdot\xi}dx.$$ We use $1_U$ to denote the
indicator function of the set $U$, i.e.,
\begin{equation*}1_U(x)=\begin{cases} 1, &\text{for } x\in U,\\
0, &\text{for } x\notin U.\end{cases}\end{equation*} For $1\le p\le \infty$, we denote the
conjugate exponent of $p$ by $p'$, i.e., $1/p+1/p'=1$.

We start with stating the main theorem concerning the linear restriction estimates, which is
proven in Section \ref{sec:linear}.

\begin{theorem}\label{thm:dyadic-lin}
Suppose $f\in \L_1$ and $R>0$ is a dyadic number. Then the following sharp restriction estimates
hold:
\begin{itemize}
\item\label{eq:loc-24} for $q=2$ and $ 2\le p\le \infty$,
$$\|(fd\sigma)^{\vee}\|_{L^2_{t,x}}\lesssim
\min\{R^{\frac 12}, R^{\frac{n-1}{2}}\}\|f\|_{L^p(S)}.$$

\item for $q=4$ and $4\le p\le \infty$, $\forall \, \eps>0$,
$$\|(fd\sigma)^{\vee}\|_{L^4_{t,x}}\lesssim_{\eps}
\min\{R^{-\frac{n-2}{4}+\eps}, R^{\frac{n-1}{4}}\}\|f\|_{L^p(S)}.$$

\item for $q=3p'$ and $1\le p<4$,
$$\|(fd\sigma)^{\vee}\|_{L^q_{t,x}}\lesssim
\min\{R^{(n-2)(\frac 1q-\frac12)}, R^{\frac{n-1}{q}}\}\|f\|_{L^p(S)}.$$

\item for $q=\infty$ and $1\le p\le\infty$,
$$\|(fd\sigma)^{\vee}\|_{L^{\infty}_{t,x}}\lesssim
\min\{R^{-\frac {n-2}{2}}, 1\}\|f\|_{L^p(S)}.$$
\end{itemize}
where $A$ in $\{A, B\}$ is given by the case where $R\ge 2$ and $B$ by $R\le 1$. Furthermore, by
interpolation we obtain the sharp restriction estimates $L^p\to L^q$ when $p, q$ are in the region
determined by these lines (Figure \ref{fig:lin-dya}).
\end{theorem}

\begin{figure}[ht]
 \begin{center}
\includegraphics[scale=0.6]{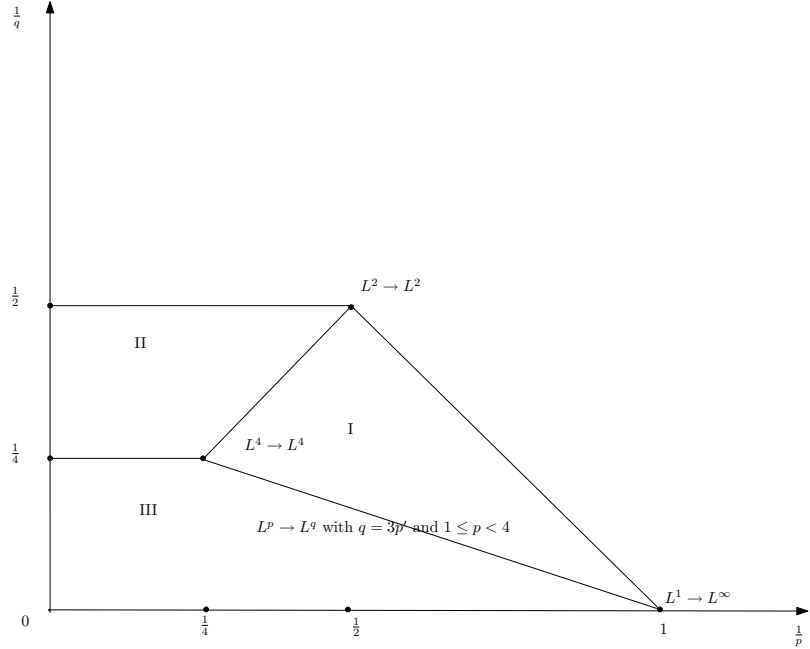}\\
  \caption{Linear restriction estimates on $\R\times A_R$.}\label{fig:lin-dya}
  \end{center}
\end{figure}

\begin{remark}
We observe that the estimates above in each case are ``continuous" in the sense that they match
when $R\sim 1$.
\end{remark}

One can easily obtain the following corollary regarding the linear adjoint restriction conjecture.
\begin{corollary}\label{cor:lin-restr} Suppose $f$ are
cylindrically functions supported on the paraboloid.
\begin{itemize}
\item If $f\in \L_M$, the linear adjoint restriction conjecture \ref{con:lin-restr} holds with
constants depending on $p,\,q,\,n$ and $M$ whenever $q>\frac {2n}{n-1}$, $\frac {1}{p}+\frac
{1}{q}\le 1$ and $ \frac {1}{p}+\frac {2n-1}{q}< n-1$ (Figure \ref{fig:lin}).
 \item The \textit{linear adjoint restriction conjecture}
 \ref{con:lin-restr} is true for all $f$ (Figure \ref{fig:lin}).
 \end{itemize}
 \end{corollary}
 \begin{proof}
By Theorem \ref{thm:dyadic-lin}, the first assertion follows from scaling to $f\in \L_1$ and then
summing all dyadic $R$ and using interpolation.

To prove the second assertion, it is sufficient to obtain it under the boundary conditions
$q>\frac{2n}{n-1}$ and $\frac{n+1}{q}=\frac{n-1}{p'}$ since other estimates are easily obtained by
a standard argument of using the H\"older inequality. By interpolating between the $L^2 \to L^2 $
estimate and the $L^p\to L^q$ estimates on the line segment $q=3p'$ and $1\le p<4$, we obtain
that, for $q>\frac{2n}{n-1}$, $\frac{n+1}{q}=\frac{n-1}{p'}$ and $f\in \L_1$,
\begin{equation*}
\|(fd\sigma)^{\vee}\|_{L^{q}_{t,x}}\lesssim R^{\alpha(R)}\|f\|_{L^p(S)},
\end{equation*} where $\alpha$ is the step function
\begin{equation*}
\alpha(R)=\begin{cases}
         -\frac{n-2}{2}[1-\frac{2n}{q(n-1)}], &\text{for } R\ge 2, \\
         \frac{n-1}{q}, &\text{for } 0<R\le 1.
\end{cases}
\end{equation*} We remark that the constant above $C R^{\alpha(R)}$ does not depend on $\eps$.
By scaling, when $f\in \L_M$, under the same conditions on $p$ and $q$,
\begin{equation*}
\|(fd\sigma)^{\vee}\|_{L^{q}_{t,x}}\lesssim (RM)^{\alpha(RM)}\|f\|_{L^p(S)},
\end{equation*}
where $\alpha$ is defined as above. Then for all cylindrically symmetric $f$ supported on the
paraboloid, we decompose it into a sum of dyadically supported functions,
\begin{equation*}
f=\sum_{M: \,dyadic} f1_{\{(\tau, \xi): \tau=|\xi|^2, M\le |\xi|\le 2M\}}=\sum_{M} f_M,
\end{equation*} where $f_M:=f1_{\{(\tau, \xi): \tau=|\xi|^2, M\le |\xi|\le 2M\}}$.
Hence
\begin{align*}
\|(fd\sigma)^{\vee}\|_{L^q_{t,x}(\R\times \R^{n-1})}&
=\left(\sum_{R}\|(fd\sigma)^{\vee}\|^q_{L^q_{t,x}(\R\times A_R)}\right)^{1/q}\\
&=\left( \sum_{R}\|\sum_{M}(f_Md\sigma)^{\vee}\|^q_{L^q_{t,x}(\R\times A_R)}\right)^{1/q}\\
&\le \left(\sum_{R}\left( \sum_{M}\|(f_Md\sigma)^{\vee}\|_{L^q_{t,x}(\R\times A_R)}\right)^q
\right)^{1/q}\\
&\lesssim \left(\sum_{R}\left(\sum_{M}(RM)^{\alpha(RM)}\|f_M\|_{L^p(S)}\right)^q\right)^{1/q}\\
&\lesssim \left(\sum_{M}\|f_M\|^p_{L^p(S)}\right)^{1/p}=\|f\|_{L^p(S)},
\end{align*} where $R>0$ and $M>0$ are both dyadic numbers; in the last inequality, since
\begin{align*}
\forall R>0,\,
\sum_{M}(RM)^{\alpha(RM)}< \infty,\\
\forall M>0, \,\sum_{R}(RM)^{\alpha(RM)}< \infty,
\end{align*} we have used the Schur's test for exponents $p$ and $q$ satisfying the condition
$q>\frac{2n}{n-1}>p\ge 1$.
\end{proof}

\begin{remark}
In the cylindrically symmetric case, we remark that $q>\frac{2n}{n-1}$ is still sharp since it is
given by the decay estimate $|(d\sigma)^{\vee}(t,\xi)|\le C_n(1+|t|+|\xi|)^{(1-n)/2}$, see e.g.,
\cite[Chapter 8, Theorem 3.1]{Stein:1993}.
\end{remark}

\begin{figure}[ht]
\begin{center}
 \includegraphics[scale=0.6]{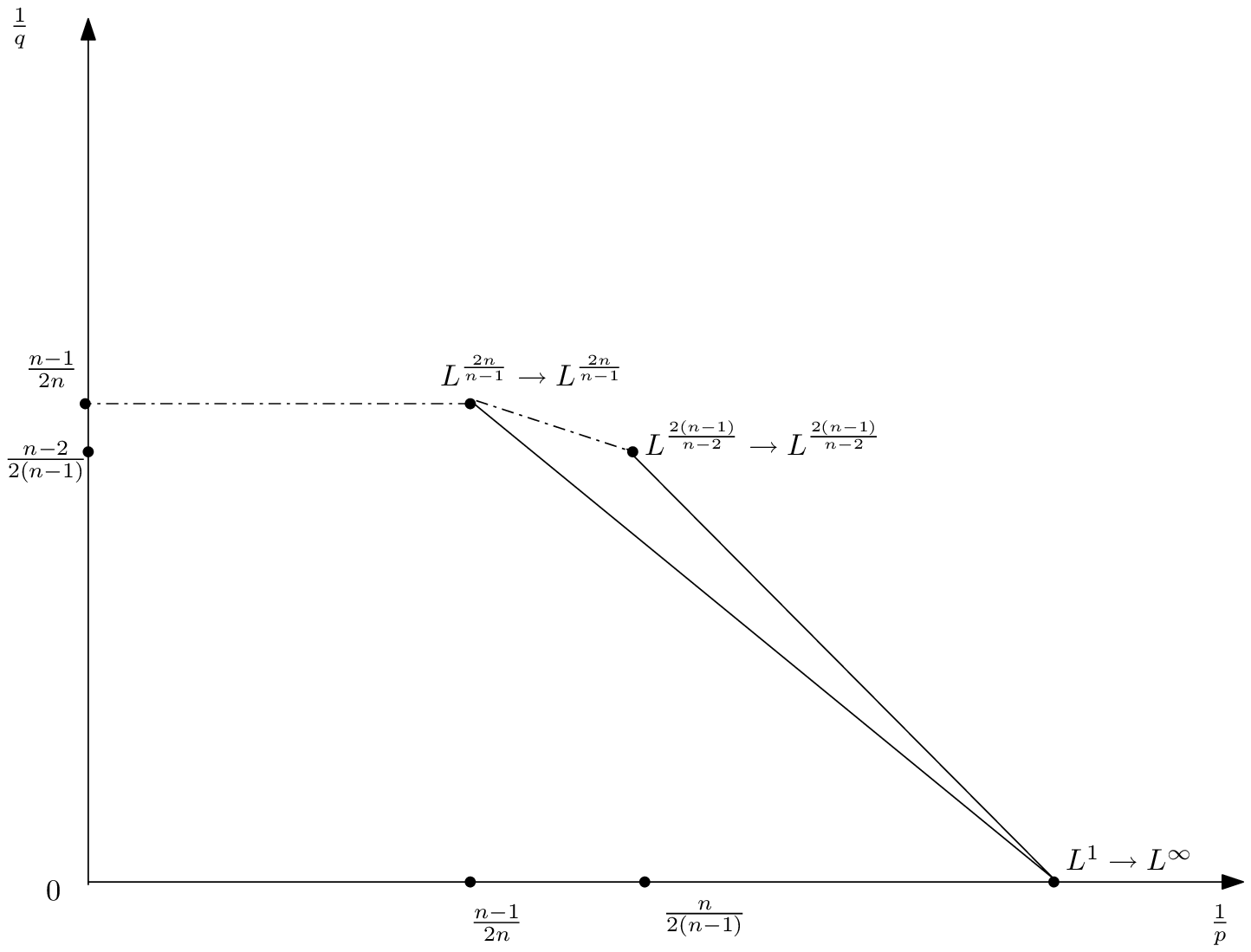}\\
  \caption{Linear restriction estimates for $\L_M$ (cf.~the inside trapezoid, the
   classical range of $p$ and $q$ for Conjecture \ref{con:lin-restr}).}\label{fig:lin}
\end{center}
\end{figure}
Next we state the theorem regarding the bilinear restriction estimates in the cylindrically
symmetric case, which is proven in Section \ref{sec:biliear}.
\begin{theorem}\label{thm:dyadic-bilin}
Suppose $f\in \L_1$ and $g\in \L_M$ with $0<M\le 1/4$. $R>0$ is a dyadic number. Then the
following sharp bilinear restriction estimates hold:
\begin{itemize}
\item for $q=1$ and $ 2\le p\le \infty$,
\begin{align*}
\|&(fd\sigma_1)^{\vee} (gd\sigma_2)^{\vee}\|_{L^1_{t,x}}\nonumber\\
&\lesssim \min\{RM^{\frac{n-2}{2}-\frac{n-1}{p}},\,R^{\frac{n}{2}}M^{-1+\frac{n-1}{p'}},\,R^{n-1}
M^{-1+\frac{n-1}{p'}}\}\|f\|_{L^p(S_1)}\|g\|_{L^p(S_2)}.
\end{align*}
\item for $q=2$ and $2\le p\le \infty$,
\begin{align*}
\|&(fd\sigma_1)^{\vee} (gd\sigma_2)^{\vee}\|_{L^2_{t,x}}\nonumber\\
&\lesssim
\min\{R^{-\frac{n-2}{2}}M^{\frac{n-1}{2}-\frac{n-1}{p}},\,R^{\frac{1}{2}}M^{\frac{n-1}{p'}},
\,R^{\frac{n-1}{2}}M^{\frac{n-1}{p'}}\}\|f\|_{L^p(S_1)}\|g\|_{L^p(S_2)}.
\end{align*}
\item for $q=4$ and $4\le p\le \infty$, $\forall \, \eps>0$,
\begin{align*}&\|(fd\sigma_1)^{\vee} (gd\sigma_2)^{\vee}\|_{L^4_{t,x}}\\
&\lesssim_{\eps}\min\{R^{-\frac{3(n-2)}{4}+\eps}M^{\frac{n}{2}-\frac{n-1}{p}},\,R^{-\frac{n-2}{4}+\eps}
M^{\frac{n-1}{p'}},\,R^{\frac{n-1}{4}}M^{\frac{n-1}{p'}}\}\|f\|_{L^p(S_1)}\|g\|_{L^p(S_2)}.
\end{align*}
\item for $q=3p'$ and $1\le p<4$,
\begin{align*}
&\|(fd\sigma_1)^{\vee} (gd\sigma_2)^{\vee}\|_{L^q_{t,x}}\\
&\lesssim\min\{R^{-\frac{n-2}{q'}}M^{\frac{n}{2}-\frac{n-1}{p}},\,R^{(n-2)(\frac1q-\frac{1}{2})}
M^{\frac{n-1}{p'}},\,R^{\frac{n-1}{q}}M^{\frac{n-1}{p'})}\}\|f\|_{L^p(S_1)}\|g\|_{L^p(S_2)}.
\end{align*}
\item for $q=\infty$ and $1\le p\le\infty$,
\begin{align*}
&\|(fd\sigma_1)^{\vee}(gd\sigma_2)^{\vee}\|_{L^{\infty}_{t,x}} \\
&\lesssim\min\{R^{-(n-2)}M^{\frac{n}{2}-\frac{n-1}{p}},\,R^{-\frac{n-2}{2}}
M^{\frac{n-1}{p'}},\,M^{\frac{n-1}{p'}}\}\|f\|_{L^p(S_1)}\|g\|_{L^p(S_2)}.
\end{align*}
\end{itemize}
where $A$ in $\{A, B, C\}$ is given by the case where $R\ge 1/M$, B by $2\le R\le 1/M$ and $C$ by
$R\le 1$. Furthermore, by interpolation we obtain the sharp restriction estimates $L^p\times
L^p\to L^q$ when $p$ and $q$ are in the region determined by these lines (Figure \ref{fig:
bilin-dya}).
\end{theorem}

\begin{figure}[ht]
\begin{center}
\scalebox{0.6}[0.5]{\includegraphics{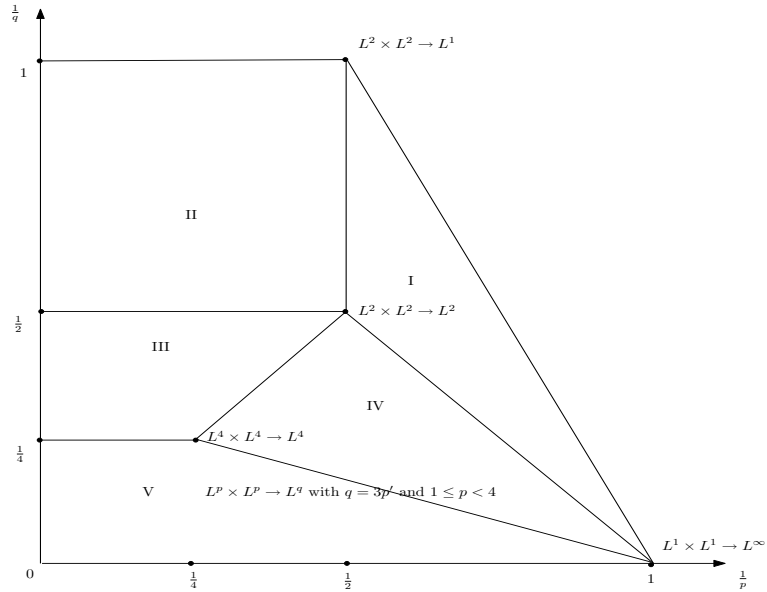}}\\
\caption{Bilinear restriction estimates on $\R\times A_R$.}\label{fig: bilin-dya}
\end{center}
\end{figure}
\begin{remark}
We observe that $R\ge \frac1M$ if and only if $A\lesssim B$ and $R\ge 1$ if and only if $B\lesssim
C$. In other words, the estimates are ``continuous'' in the sense that $A\sim B$ when $R\sim 1/M$
and $B\sim C$ when $R\sim 1$.
\end{remark}
As a corollary of Theorem \ref{thm:dyadic-bilin}, we see that the \emph{bilinear adjoint
restriction conjecture} holds for exponents $p$ and $q$ in a larger region.
\begin{corollary}\label{cor:bilin-restr} Suppose $f$ and $g$ are defined as Theorem
\ref{thm:dyadic-bilin}. Then the \textit{bilinear adjoint restriction conjecture}
\ref{con:bilin-restr} holds with constants depending on $p,\,q,\,n,\, S_1,\, S_2$ and $M$,
whenever $q>\frac {n}{n-1}$, $ \frac {2}{p}+\frac {n}{q}<n$ and $\frac {2}{p}+\frac {1}{q}\le 2$.
These estimates are sharp except for certain endpoints (Figure \ref{fig:bilin}).
\end{corollary}
\begin{figure}[ht]
\begin{center}
\includegraphics[scale=0.6]{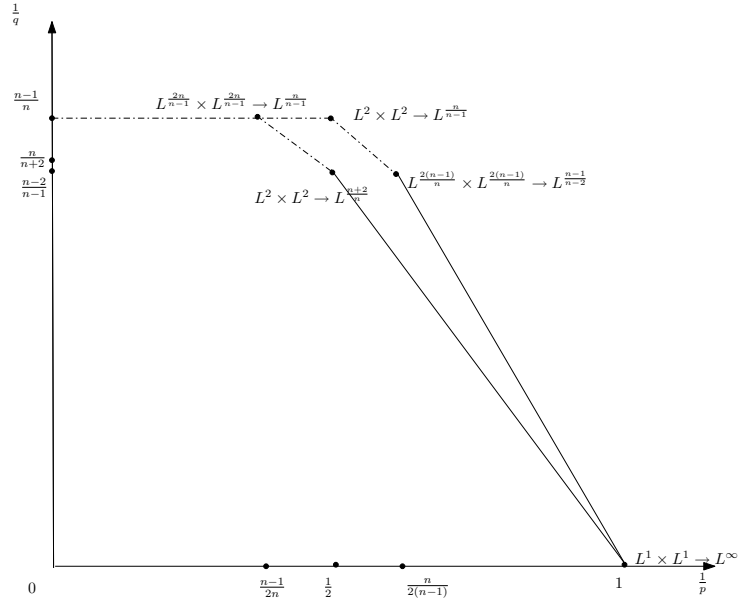}\\
\caption{Bilinear restriction estimates for functions in $\L_1$ and $\L_M$ with $0<M\le 1/4$ (cf.
~the inside pentagram, the classical range of $p$ and $q$ for Conjecture
\ref{con:bilin-restr}).}\label{fig:bilin}
\end{center}
\end{figure}
\begin{remark}
When $S$ is replaced by the lower third of the sphere, the analogous results to Theorems
\ref{thm:dyadic-lin} and \ref{thm:dyadic-bilin} hold, which will be accomplished in Section
\ref{sec:other-surf}. This is essentially due to the common geometric property of non-vanishing
Gaussian curvature shared by the sphere and the paraboloid and the fact that the sphere locally
resembles the paraboloid, which can be seen from the Taylor expansion $\sqrt{1-|\xi|^2}\sim
1-c(|\xi|)|\xi|^2$ when $|\xi|$ is small.
\end{remark}
\begin{remark}
It is well known that the adjoint restriction estimates are closely related to the Strichartz
estimates for the nonlinear dispersive equations such as the Schr\"odinger equation and the wave
equation, see e.g., \cite{Strichartz:1977} and \cite{Stein:1993}. We will establish this
connection in our case in Section \ref{sec:PDE appli}.
\end{remark}

This paper is organized as follows: Section \ref{sec:linear} is devoted proving Theorem
\ref{thm:dyadic-lin} and constructing counterexamples to show the linear estimates are sharp.
Section \ref{sec:biliear} is devoted proving Theorem \ref{thm:dyadic-bilin} and constructing
counterexamples showing the bilinear estimates are sharp. In Section \ref{sec:other-surf} we will
establish analogous results for the cylindrically symmetric functions supported on the lower third
of the sphere or the cylindrically symmetric and compact hypersurfaces of elliptic type. In
Section \ref{sec:PDE appli} we will interpret the restriction estimates in terms of solutions to
the Schr\"odinger equation to establish the Strichartz estimates.

\section{Proof of Theorem \ref{thm:dyadic-lin}: linear estimates and
examples}\label{sec:linear}

For any cylindrically symmetric function $f$ on the paraboloid, we set $F(|\xi|)=f(|\xi|^2, \xi)$.
We observe that $(fd\sigma)^{\vee}(t, x)$ is also a cylindrically symmetric function. To begin the
proof of Theorem \ref{thm:dyadic-lin}, we first investigate the behavior of $(fd\sigma)^{\vee}$ on
$\{|x|\le 1\}$ via the following proposition.
\begin{proposition}\label{prop: sml-err}
Suppose $f\in \L_1$. Then for any $1\le p\le \infty$, $q\ge \max\{2, p'\}$ and $R\le 1$, we have a
sharp estimate
\begin{equation}\label{eq:loc-1}
\|(fd\sigma)^{\vee}\|_{L^q_{t,x}} \lesssim R^{\frac{n-1}{q}}\|f\|_{L^p(S)}.
\end{equation}
\end{proposition}

\begin{proof}
If we change to polar coordinates, the left-hand side of \eqref{eq:loc-1} is
\begin{align*}
\|(fd\sigma)^{\vee}&\|_{L^q_{t,x}}=\left(\int_{R/2}^{R}\int_{\R} \left|
\int_{1\le|\xi|\le2}f(\xi) e^{i(x\cdot \xi-t\cdot |\xi|^2)}d\,\xi\right|^q d\,t\,d\,x \right)^{1/q}\\
&= \left(\int_{R/2}^{R}\int_{\R}\left|\int_{I}F(s)s^{n-2}e^{-its^2}\int_{\S^{n-2}}e^{irs\omega}
d\mu(\omega)ds\right|^q dt\,r^{n-2}dr \right)^{1/q},
\end{align*} where $I=[1,2]$. Then we change variables back $s\to a^{1/2}$ to majorize it by
\begin{equation*}
\left(\int_{R/2}^{R}\int_{\R}\left|\int_{I'}
a^{(n-3)/2}F(a^{1/2})(d\mu)^{\vee}(ra^{1/2}e_1)e^{-ita}da\right|^q dt\,r^{n-2}dr \right)^{1/q},
\end{equation*} where $I'=[1, \sqrt2]$ and
$e_1=(1,0,\ldots,0)\in \R^{n-1}$.

Then by the Hausdorff-Young inequality when $q>2$ or the Plancherel theorem when $q=2$, changing
$s\to a=s^2$ and the fact that $\|(d\mu)^{\vee}\|_{L^{\infty}_{\omega}}\lesssim 1$, the left-hand
side of \eqref{eq:loc-1} is further bounded by
\begin{equation*}
R^{\frac{n-2}{q}}\left(\int_{R/2}^R \|F\|^q_{L^{q'}(I)}\,dr\right)^{1/q}\sim
R^{\frac{n-1}{q}}\|F\|_{L^{q'}(I)}.
\end{equation*} Then by  the H\"older inequality and the fact $\|F\|_{L^{p}(I)}\sim \|f\|_{L^{p}(S)}$,
\eqref{eq:loc-1} follows.

Now we will construct a counterexample to show the estimate \eqref{eq:loc-1} is sharp when $1\le
p\le \infty$ and $q\ge \max\{2,p'\}$. We take
$$f(|\xi|^2, \xi)=F(|\xi|)=|\xi|^{-(n-2)}1_{\{1\le |\xi|\le 2\}}
e^{it_0|\xi|^2},$$ where $t_0\in \R$. Then the left-hand side of \eqref{eq:loc-1} reduces to
$$\left(\int_{R/2}^{R}\int_{\R}\left|\int_{I}
e^{-i(t-t_0)s^2}\int_{\S^{n-2}}e^{irs\omega}d\mu(\omega)ds\right|^q dt\,r^{n-2}dr \right)^{1/q}.$$
We choose $r$ and $t$ satisfying that
$$r\in [R/100,R/50],\, |t-t_0|\le 1/100.$$ Then the left-hand side of
\eqref{eq:loc-1} $\gtrsim R^{\frac{n-1}{q}}$ and its right-hand side $\lesssim R^{\frac{n-1}{q}}$.
Hence \eqref{eq:loc-1} is easily seen to be sharp.
\end{proof}

Before investigating the behavior of $(fd\sigma)^{\vee}$ on $|x|\ge 1$, we shall exploit the
spatial rotation-invariant symmetry of $f$ in the following proposition.
\begin{lemma}[Fourier-Bessel formula]\label{lem:Four-Bess}
Suppose $f$ is a cylindrically symmetric function supported on the paraboloid. Then
\begin{align}\label{eq: Fourier-Bessel formula}
(f&d\sigma)^{\vee}(t, x)\nonumber\\
&=c_nr^{-\frac{n-2}{2}}\int_{I} F(s)s^{\frac{n-2}{2}} e^{i(rs-ts^2)}
ds+c_nr^{-\frac{n-2}{2}}\int_{I} F(s)s^{\frac{n-2}{2}}e^{-i(rs+ts^2)} ds\nonumber \\
&+c_n\int_{I}F(s)s^{n-2}e^{-its^2-irs}\int_{0}^{\infty}e^{-rsy}y^{\frac{n-4}{2}}[(y+2i)^{\frac{n-4}{2}}
-(2i)^{\frac{n-4}{2}}]dyds\nonumber\\
&-c_n\int_{I}F(s)s^{n-2}e^{-its^2+irs}\int_{0}^{\infty}e^{-rsy}y^{\frac{n-4}{2}}[(y-2i)^{\frac{n-4}{2}}
-(-2i)^{\frac{n-4}{2}}]dyds,
\end{align}
where $r=|x|$ and $I$ is the projection interval in the radial direction.
\end{lemma}
\begin{proof}
We first expand $(fd\sigma)^{\vee}$ in the polar coordinates,
\begin{equation*}
(fd\sigma)^{\vee}(t, x)=\int_{\{|\xi|\in I\}} f(|\xi|^2, \xi)e^{i(x\xi-t|\xi|^2)}d\xi
 =\int_I F(s)e^{-its^2}s^{n-2}(d\mu)^{\vee}(rse_1)ds,
\end{equation*} where $d\mu$ is the surface measure of the sphere $\S^{n-2}$.

On the one hand, the inverse Fourier transform of $d\mu$ is given by
\begin{equation*}(d\mu)^{\vee}(\xi)=c_n|\xi|^{\frac {3-n}{2}}J_{\frac{n-3}{2}}(|\xi|),\end{equation*}
where $J_{\frac {n-3}{2}}$ is the Bessel function of order $\frac {n-3}{2}$, see e.g.
\cite{Stein:1993} or \cite{Stein-Weiss:1971:fourier-analysis}. On the other hand, by using the
same argument as proving \cite[Lemma 3.11]{Stein-Weiss:1971:fourier-analysis}, we obtain that, for
fixed $m\ge 0$,
\begin{align*}
J_m(r)&=c_m \frac{e^{ir}-e^{-ir}}{r^{1/2}}+c_mr^me^{-ir}\int_{0}^{\infty}e^{-ry}y^{\frac{2m-1}{2}}
[(y+2i)^{\frac{2m-1}{2}}-(2i)^{\frac{2m-1}{2}}]d\,y\\
&-c_mr^me^{ir}\int_{0}^{\infty}e^{-ry}y^{\frac{2m-1}{2}}[(y-2i)^{\frac{2m-1}{2}}-
(-2i)^{\frac{2m-1}{2}}]d\,y.
\end{align*} Hence \eqref{eq: Fourier-Bessel formula} holds after we combine these two estimates and set
$m=\frac {n-3}{2}$.
\end{proof}

Therefore we define the error term of $(fd\sigma)^{\vee}$ by
\begin{align*}
\mathcal{E}&f(t,x):=\pm c_n\int_{I}F(s)s^{n-2}e^{-its^2\mp irs}\times\\
&\times\int_{0}^{\infty}e^{-rsy}y^{\frac{n-4}{2}}[(y\pm2i)^{\frac{n-4}{2}}
-(\pm2i)^{\frac{n-4}{2}}]dyds,
\end{align*} where by $\pm$ we denote a sum of two terms where $+$ and $-$ appear alternatively.
Heuristically, one should think of $\E f$ as a term comparable to
$r^{-n/2}\int_{I}F(s)s^{\frac{n-4}{2}}e^{-its^{2}}ds$, which comes from estimating the error term
of Bessel function $J_m(r)$ by $r^{-3/2}$. At the first approximation, this simplification is easy
to deal with and intuitively provides what the bound involving the error term should be. However
in this paper we will establish it rigorously in the following proposition, which shows that the
information about its contribution to the linear estimates when $|x|\ge 1$. It is acceptable
compared to the main term estimates established in the next propositions.
\begin{proposition}\label{prop: lin-err J_m}
Suppose $f\in \L_1$. Then for any $q\ge \max\{2, p'\}$ and any dyadic number $R\ge 2$ and $f\in
L^p(S)$,
\begin{equation} \label{eq:loc-3}
\|\mathcal{E}f\|_{L^q_{t,x}}\lesssim R^{-\frac n2+\frac{n-1}{q}}\|f\|_{L^{p}({S})}.
\end{equation}
\end{proposition}
\begin{proof}
We set
$$E(r):=\int_{0}^{\infty}e^{-ry}y^{\frac{n-4}{2}}[(y\pm2i)^{\frac{n-4}{2}}-
(\pm2i)^{\frac{n-4}{2}}]dy.$$
For $r\ge 1$, we first estimate $E(r)$ by repeating the proof of \cite[Lemma
3.11]{Stein-Weiss:1971:fourier-analysis} for readers' convenience.
\begin{align*}
|E(r)|\le&\int_0^1e^{-ry}y^{\frac{n-4}{2}}|(y\pm2i)^{\frac{n-4}{2}}-(\pm2i)^{\frac{n-4}{2}}|dy+\\
&+\int_{1}^{\infty}e^{-ry}y^{\frac{n-4}{2}}|(y+2i)^{\frac{n-4}{2}}-(2i)^{\frac{n-4}{2}}|dy\\
&=:I+II
\end{align*}
For $I$, where $0\le y\le 1$, by the mean value theorem we have
$|(y\pm2i)^{\frac{n-4}{2}}-(\pm2i)^{\frac{n-4}{2}}|\lesssim y$. For $II$, where $y\ge 1$, we take
$y$ out and then use the mean value theorem to obtain
$|(y\pm2i)^{\frac{n-4}{2}}-(\pm2i)^{\frac{n-4}{2}}|\lesssim y^{\frac{n-4}{2}}$. Then combining
these estimates above,
\begin{align*}
|E(r)|&\le\int_0^1e^{-ry}y^{\frac{n-2}{2}}dy+\int_{1}^{\infty}e^{-ry}y^{n-4}dy\\
&\lesssim r^{-n/2}\int_0^re^{-y}y^{\frac n2-1}dy+r^{-(n-3)}\int_r^{\infty}e^{-y}y^{n-4}dy.
\end{align*} By the definition of the Gamma function $\Gamma$,
$ r^{-n/2}\int_0^re^{-y}y^{\frac n2-1}dy\lesssim \Gamma(n/2)r^{-n/2}$. Then using integration by
parts $n-4$ times when $n\ge 4$ and $y^{-1}\le r^{-1}$ when $n=3$, we obtain $
r^{-(n-3)}\int_r^{\infty}e^{-y}y^{n-4}dy\lesssim e^{-r}r^{-1}$.

Since $e^{-r}r^{-1+\frac n2}$ is continuous on $[1,\infty)$ and decays to $0$ as $r\to\infty$,
$e^{-r}r^{-1}\lesssim r^{-\frac n2}$ holds for $r\ge 1$. Therefore
\begin{equation}\label{eq:error for J_m} |E(r)|\lesssim r^{-n/2}.
\end{equation}
Next let us turn to the estimate \eqref{eq:loc-3}. By changing to polar coordinates, the left-hand
side of  \eqref{eq:loc-3} is comparable to
\begin{equation} \label{eq:loc-2}
\left(\int_{R/2}^R\int_{\R}\left|\int_{I}F(s)s^{n-2}e^{-its^2\mp irs}E(rs)\,ds
\right|^qdt\,r^{n-2}dr\right)^{1/q},
\end{equation} where $I=[1,2]$. After changing variables $s=a^{1/2}$, \eqref{eq:loc-2} is
comparable to
\begin{equation*}
\left(\int_{R/2}^R\int_{\R}\left|\int_{I'}F(a^{1/2})a^{\frac{n-3}{2}}E(ra^{1/2})e^{\mp
ira^{1/2}}e^{-ita}\,da\right|^qdt\,r^{n-2}dr\right)^{1/q},
\end{equation*} where $I'=[1, \sqrt2]$. Then by the Hausdorff-Young inequality when $q>2$ or
the Plancherel theorem when $q=2$, changing variables back $a=s^2$ and $s\sim 1$, the left-hand
side of \eqref{eq:loc-2} is further majorized by \begin{equation*}
\left(\int_{R/2}^R\int_{\R}\left|\int_{I}\left|F(s)E(rs)\right|^{q'}ds
\right|^{q/q'}dt\,r^{n-2}dr\right)^{1/q}.
\end{equation*}
Since $rs\ge 1$ for any $r\in [R/2, R]$ and $s\in I$, \eqref{eq:error for J_m} and the H\"older
inequality give $R^{-\frac n2+\frac{n-1}{q}}\|F\|_{L^p(I)}$. Since $\|F\|_{L^p(I)}\sim
\|f\|_{L^p(S)}$, \eqref{eq:loc-3} follows.
\end{proof}

When $|x|\ge 1$, we are left with estimating the main term of $(fd\sigma)^{\vee}$,
\begin{equation} \label{prop: heur-four-bess}
\M f(t,x):=c_nr^{-\frac{n-2}{2}}\int_{I}F(s)s^{\frac{n-2}{2}} e^{i(\pm rs-ts^2)} ds,
\end{equation} where by $\pm$ we denote the summation of two terms. We call it \textit{the heuristic
approximation of $(fd\sigma)^{\vee}$}. We are going to prove the positive part ``estimates" of
Theorem \ref{thm:dyadic-lin} in the following four propositions. In the remainder of this section,
we will prove its negative part ``sharpness" by certain counterexamples.

\begin{proposition}[$q=2$ line]\label{prop: lin-q=2}
Suppose $f\in \L_1$. Then for $ q=2$, $ 2\le p\le \infty$ and $R\ge 2$, we have a sharp estimate
\begin{equation}\label{eq:linear q=2}
\|(fd\sigma)^{\vee}\|_{L^2_{t,x}}\lesssim R^{1/2}\|f\|_{L^p(S)}.
\end{equation}
\end{proposition}
\begin{proof}
We observe that it is sufficient to estimate the term of $\M f$ with $+$ sign. In the propositions
followed, we will make the same reduction unless specified. Hence by the heuristic approximation
\eqref{prop: heur-four-bess} of $(fd\sigma)^{\vee}$ with $+$ sign, changing variables and then the
Plancherel theorem in $t$ followed by the H\"older inequality,
\begin{align*}
\|(fd\sigma)^{\vee}\|_{L^2_{t,x}}&=\left(\int_{R/2\le |x|\le R}\int_{\R}
|(fd\sigma)^{\vee}(t,x)|^2 d\,t\, d\,x \right)^{1/2}\\
&\sim \left(\int_{R/2}^R\int_{\R}\left|r^{-\frac{n-2}{2}}\int_{I}F(s)s^{\frac{n-2}{2}}
e^{i[rs-ts^2]}ds\right|^2dt\, r^{n-2}dr\right)^{1/2}\\
&=\left(\int_{R/2}^R\int_{\R}\left|\int_{I}F(s)s^{\frac{n-2}{2}}e^{irs}e^{-ts^2}ds\right|^2
dt\, dr\right)^{1/2}\\
&=\left(\int_{R/2}^R\|F\|^2_{L^2_s(I)}dr\right)^{1/2}\lesssim R^{1/2}\|f\|_{L^2(S)}\lesssim
R^{1/2}\|f\|_{L^p(S)},
\end{align*} where $I=[1,2]$, $2\le p\le \infty$. Hence \eqref{eq:linear q=2} follows.
\end{proof}

Now let us deal with the estimates on the line $q=4$. The estimate $L^4\to L^4$ is the endpoint of
two dimensional ($n=2$) linear adjoint restriction conjecture and hence the classical $TT^*$
approach, namely the Carleson-Sj\"olin argument used in Proposition \ref{prop:lin-q=3p'},
unfortunately fails because we can not apply the Hardy-Littlewood-Sobolev inequality at one step.
Instead, we can perform a Whitney-type decomposition to $I$ to create some frequency separation.
Similar arguments can be found in \cite{Tao:2003:paraboloid-restri},
\cite{Tao-Vargas:2000:cone-2}, \cite{Tao-Vargas-Vega:1998:bilinear-restri-kakeya}.

\begin{proposition}[$q=4$ line]\label{prop:lin-q=4}
Suppose $f\in \L_1$. Then for $q=4$, $4\le p\le \infty$ and $\eps>0$ and $R\ge 2$, we have a sharp
estimate up to $R^{\eps}$,
\begin{equation}\label{eq:lin-q=4}
\|(fd\sigma)^{\vee}\|_{L^4_{t,x}}\lesssim_{\eps} R^{-(n-2)/4+\eps}\|f\|_{L^p(S)}.
\end{equation}
\end{proposition}

\begin{proof}
By the same reasoning as that used in the proof of Proposition \ref{prop: lin-q=2}, we will only
prove the estimate when $p=4$. By the heuristic approximation \eqref{prop: heur-four-bess} of
$(fd\sigma)^{\vee}$ with $+$ sign,
$$\|(fd\sigma)^{\vee}\|^2_{L^4_{t,x}} \sim R^{-(n-2)/2}\left(\int_{R/2}^R\int_{\R}
\left|\int_{I}F(s)s^{\frac{n-2}{2}}e^{i(rs-ts^2)}ds\right|^4dt\,dr\right)^{1/2},$$ where
$I=[1,2]$.

We set $(Fd\sigma)^{\vee}(t,r):=\int_{I}F(s)s^{\frac{n-2}{2}}e^{i[rs-ts^2]}ds$, which can be
regarded as the inverse space-time Fourier transform of $f(\tau, \xi)|\xi|^{\frac{n-2}{2}}$
restricted to the parabola. To prove \eqref{eq:lin-q=4}, it suffices to prove that, for any
$\eps>0$,
\begin{equation}\label{eq:loc-33}
\left(\int_{R/2}^R\int_{\R}\left|(Fd\sigma)^{\vee}(Fd\sigma)^{\vee}\right|^2d\,t\,dr\right)^{1/2}
\lesssim_{\eps} R^{\eps}\|F\|^2_{L^4}.
\end{equation}
Next we will perform a Whitney-type decomposition to $I=[1,2]$. For each $j\ge 0$ we break up $I$
into $O(2^j)$ dyadic intervals $\tau^j_k$ of length $2^{-j}$, and write $\tau^j_k \backsimeq
\tau^j_{k'}$ if $\tau^j_k$ and $\tau^j_{k'}$ are not adjacent but have adjacent parents. For each
$j\ge 0$, let $F=\sum F^j_k$ where $F^j_k=F 1_{\tau^j_k}$. Then
$$(Fd\sigma)^{\vee}(Fd\sigma)^{\vee}=\sum_{j}\sum_{k,k': \tau^j_k \backsimeq \tau^j_{k'} }
(F^j_k d\sigma)^{\vee}(F^j_{k'} d\sigma)^{\vee}.$$ From the triangle inequality, the left-hand
side of \eqref{eq:loc-33} is bounded by
\begin{align*}
 \sum _{2^j\le R}&\|\sum_{k,k': \tau^j_k \backsimeq \tau^j_{k'}}(F^j_k d\sigma)^{\vee}
 (F^j_{k'} d\sigma)^{\vee}\|_{L^2_{t,r}(\R\times \R)}+\\
 &+\sum _{2^j\ge R}\sum_{k,k': \tau^j_k\backsimeq \tau^j_{k'}}
 \|(F^j_k d\sigma)^{\vee}(F^j_{k'} d\sigma)^{\vee}\|_{L^2_{t,r}(\R\times A_R)}=:A+B.
\end{align*}
We will first estimate $A$. By the quasi-orthogonality property of functions among $(F^j_k
d\sigma)^{\vee}(F^j_{k'} d\sigma)^{\vee}$\cite[Lemma
6.1]{Tao-Vargas-Vega:1998:bilinear-restri-kakeya},
$$A\lesssim \sum _{2^j\le R}\left(\sum_{k,k': \tau^j_k \backsimeq \tau^j_{k'}}\|(F^j_k d\sigma)^{\vee}
(F^j_{k'} d\sigma)^{\vee}\|^2_{L^2_{t,r}(\R\times \R)}\right)^{1/2}.$$ By using the Plancherel
theorem and the Cauchy-Schwarz inequality and $s_i\sim 1$ for $i=1,2$,
$$\|(F^j_k d\sigma)^{\vee}(F^j_{k'} d\sigma)^{\vee}\|^2_{L^2_{t,r}(\R\times \R)}
\lesssim\|F^j_k\|^2_{L^2(I_k^j)}\|F^j_{k'}\|^2_{L^2(I^{j}_{k'})}
\|d\sigma^j_k*d\sigma^j_{k'}\|_{L^{\infty}},$$ where $d\sigma^j_k$ and $d\sigma^j_{k'}$ are two
arc measures of the parabola $\{\tau=|\xi|^2\}$ in $\R\times \R$ supported on $\tau^j_k$ and
$\tau^j_{k'}$, respectively.

On the one hand, from the geometric properties of the paraboloid,
$$\|d\sigma^j_k*d\sigma^j_{k'}\|_{L^{\infty}}\lesssim 2^j.$$
On the other hand, by the H\"older inequality, $\|F^j_k\|_{L^2(I_k^j)}\le
2^{-j/4}\|F^j_k\|_{L^4(I^j_{k})}$. Thus after combining these estimates,
$$A\lesssim \sum_{2^j\le R}\left(\sum_{k,k':\tau^j_k\backsimeq
\tau^j_{k'}}\|F^j_k\|^2_{L^4(I_k^j)}\|F^j_{k'}\|^2_{L^4(I_{k'}^j)}\right)^{1/2}.$$ We also observe
that for each $k$, there are $O(1)$ $k'$ such that $\tau^j_k\backsimeq \tau^j_{k'}$. Hence by the
Cauchy-Schwarz inequality, for any $\eps>0$, $$A\lesssim (\log{R})\|F\|^2_{L^4}\lesssim_{\eps}
R^{\eps} \|F\|^2_{L^4}.$$ Next let us estimate $B$. On the one hand, by the Cauchy-Schwarz
inequality,
$$\|(F^j_{k'} d\sigma)^{\vee}\|_{L^{\infty}_{t,x}}\lesssim 2^{-j/2}\|F^j_{k'}\|_{L^2(\tau^{j}_{k'})}.$$
On the other hand, by the Plancherel theorem in $t$,
$$\|(F^j_k d\sigma)^{\vee}\|_{L^2_{t,r}(\R\times A_R)}=\left(\int_{R/2}^R\int_{\R}
|(F^j_k d\sigma)^{\vee}|^2 dt\,dr\right)^{1/2}\lesssim R^{1/2}\|F^j_{k}\|_{L^2(\tau^{j}_{k})}.$$
Since there are $O(1)$ $k'$ such that $\tau^j_k\backsimeq \tau^j_{k'}$ for each $k$, by using the
Cauchy-Schwarz inequality,
$$B\lesssim R^{1/2}\sum_{2^j\ge R}\sum_{k,k':\tau^j_k\backsimeq \tau^j_{k'}} 2^{-j/2}
\|F^j_{k'}\|_{L^2(\tau^{j}_{k'})}\|F^j_{k}\|_{L^2(\tau^{j}_{k})}\lesssim R^{1/2} \sum_{2^j\ge
R}2^{-j/2} \|F\|^2_{L^2}.$$ Thus summing in $j$ and using the H\"older inequality,
\eqref{eq:loc-33} follows.
\end{proof}

In contrast to the proof of the estimate $L^4\to L^4$ in Proposition \ref{prop:lin-q=4}, the
estimates $L^p\to L^q$ when $q=3p'$ and $1\le p<4$ can be proven by the Carleson-Sj\"olin argument
or equivalently the $TT^*$ method. Such arguments can also be used to prove the non-endpoint
Strichartz estimates as in \cite{Keel-Tao:1998:endpoint-strichartz}.

\begin{proposition}[$q=3p'$ line]\label{prop:lin-q=3p'}
Suppose $f\in \L_1$. Then for $1\le p< 4$, $q=3p'$ and $R\ge 2$, we have a sharp estimate
\begin{equation}\label{eq:linear q=3p'}
\|(fd\sigma)^{\vee}\|_{L^q_{t,x}}\lesssim R^{(n-2)(1/q-1/2)}\|f\|_{L^p(S)}.
\end{equation}
\end{proposition}

\begin{proof}
By the heuristic approximation \eqref{prop: heur-four-bess} with $+$ sign,
\begin{equation*}
\|(fd\sigma)^{\vee}\|_{L^q_{t,x}}\sim R^{(n-2)(\frac1q-\frac12)}
\left(\int_{R/2}^R\int_{\R}\left|\int_ {I}F(s)s^{\frac{n-2}{2}}e^{i(rs-t s^2)}ds
\right|^qdt\,dr\right)^{1/q},
\end{equation*} where $I=[1,2]$.

Setting $(Fd\sigma)^{\vee}(t,r):=\int_{I}F(s)s^{\frac{n-2}{2}} e^{i(rs -ts^2)}ds $, we see that it
suffices to prove
\begin{equation}\label{eq:loc}
\left(\int_{R/2}^R\int_{\R}\left|\int_{I}F(s)s^{\frac{n-2}{2}}e^{i(rs-ts^2)}ds\right|^qdt\,dr
\right)^{1/q}\lesssim\|F\|_{L^p(S)}.
\end{equation}
Squaring $(Fd\sigma)^{\vee}$, we obtain
$$\{(Fd\sigma)^{\vee}(t,r)\}^2=\int_{I\times I}F(s_1)F(s_2)(s_1s_2)^{\frac{n-2}{2}}
e^{i(r\cdot (s_1+s_2)-t(s_1^2+s_2^2))}ds_1 ds_2,$$ which is an oscillatory integral with a phase
function $r(s_1+s_2)-t(s_1^2+s_2^2)$. Its Hessian is $2(s_2-s_1)$ which vanishes when $s_1=s_2$.
But we can make a change of variables $(s_1,s_2)\to (a, b)$ with $a=s_1+s_2,\,b=s_1^2+s^2_2$. It
is easy to see that the Jacobian is $2(s_2-s_1)$. Let $\Omega$ be the image in $\R\times \R$ of
$I\times I$ under such change of variables. Then $\{(Fd\sigma)^{\vee}(t,r)\}^2=\int_{\Omega}
\tilde{F}(a,b)e^{i(ra-tb)}da\,db$, where
$\tilde{F}(a,b)={F(s_1)F(s_2)(s_1s_2)^{\frac{n-2}{2}}}/{|s_1-s_2|}$ is a function of $s_1$ and
$s_2$.

Setting $q=2r'$. Since $r'>2$, by the Hausdorff-Young inequality and $s_i\sim 1$ for $i=1,2$,
\begin{align*}
\|(Fd\sigma)^{\vee}(t,r)\|_{L^q(\R\times A_R)}^{2}
&\le\|\{(Fd\sigma)^{\vee}(t,r)\}^2\|_{L^{r'}(\R\times \R)}\\
&\le\left(\int_{\Omega}|\tilde{F}(s_1,s_2)|^r ds_1ds_2\right)^{1/r}\\
&\sim \left(\int_{I^2}|F(s_1)|^r|F(s_2)|^r\frac {1}{|s_1-s_2|^{r-1}}ds_1 ds_2\right)^{1/r}\\
&=\left(\int_{I}|F(s_1)|^r\int_{I}|F(s_2)|^r\frac  {1}{|s_1-s_2|^{1-(2-r)}}ds_2
ds_1\right)^{1/r}\\
&\sim\left(\int_{I}|F(s_1)|^rI_{2-r}(|F|^r)(s_1)ds_1\right)^{1/r},
\end{align*}
where $I_{2-r}$ is the Riesz potential of order $2-r$ defined via the spatial Fourier transform by
$\widehat{I_sf}(\xi)=|\xi|^{-s}\hat{f}(\xi)$. Since $\|f\|_{L^p}\sim \|F\|_{L^p(I)}$, it then
suffices to prove that
$$\left(\int_{I}|F|^r I_{2-r}(|F|^r) ds_1\right)^{1/r}\lesssim \|F\|_{L^p(I)}^2.$$
By the H\"{o}lder inequality, we obtain
$$\int_{I}|F|^r I_{2-r}(|F|^r) ds_1\le \||F|^r\|_{L^{p/r}(I)}\|I_{2-r}(|F|^r)\|
_{L^{1/(1-r/p)}(I)}.$$ Since $\||F|^r\|_{L^{p/r}(I)}^{1/r}=\|F\|_{L^p(I)}$, it suffices to show
$\|I_{2-r}(|F|^r)\|_{L^{1/(1-r/p)}}\lesssim \||F|^r\|_{L^{p/r}(I)}$, which will follow from the
Hardy-Littlewood-Sobolev inequality. Hence the inequality \eqref{eq:loc} follows.
\end{proof}

\begin{proposition}[$q=\infty$ line]\label{prop:lin-q=infty}
Suppose $f\in \L_1$. Then for $q=\infty$, $1\le p\le\infty$ and $R\ge 2$, we have a sharp estimate
\begin{equation}\label{eq:linear q=infty}
\|(fd\sigma)^{\vee}\|_{L^{\infty}_{t,x}}\lesssim R^{-(n-2)/2}\|f\|_{L^p(S)}.
\end{equation}
\end{proposition}
\begin{proof}
By the heuristic approximation \eqref{prop: heur-four-bess} of $(fd\sigma)^{\vee}$ with $+$ sign
and the H\"older inequality, for any $p\ge 1$,
\begin{equation*}
\|(fd\sigma)^{\vee}\|_{L^{\infty}_{t,x}}\sim R^{{-(n-2)/2}}\|\int_{I}
F(s)s^{\frac{n-2}{2}}e^{i(rs-ts^2)}ds\|_{L^{\infty}_{t,r}}\lesssim R^{-(n-2)/2}\|f\|_{L^p(S)}.
\end{equation*}
\end{proof}

Now we see that the restriction estimates in Theorem \ref{thm:dyadic-lin} follow from Propositions
\ref{prop: lin-err J_m}, \ref{prop: lin-q=2}, \ref{prop:lin-q=4}, \ref{prop:lin-q=3p'} and
\ref{prop:lin-q=infty}.

The remainder of this section is devoted constructing counterexamples. In view of the propositions
above, we will use the main term of $(fd\sigma)^{\vee}$, $$\M f(t,x)=c_nr^{-\frac{n-2}{2}}\int_{I}
F(s)s^{\frac{n-2}{2}} e^{i(\pm rs-ts^2)} ds,$$ since the bound $B$ given by the error terms are
much smaller than that by the main terms when $R$ is sufficiently large .

Our first counterexample is of Knapp-type, which is designed to show the estimates in the region
$I$ in Figure \ref{fig:lin-dya} determined by the estimates $L^2\to L^2$, $L^4\to L^4$ and $L^1\to
L^{\infty}$ are sharp. The strength of the standard Knapp example or its variants lie in the idea
of using both spatial localization and frequency localization. In this example, we will only show
that the estimate $L^2\to L^2$ is sharp since the computations for others are similar.
\begin{example}[\textbf{I}]\label{ex:lin-I}
If $R\ge 2$, the $L^2\to L^2 $ estimate goes back to \eqref{eq:linear q=2} in Proposition
\ref{prop: lin-q=2}. We take
$$f(|\xi|^2, \xi)=F(|\xi|)=|\xi|^{-(n-2)/2}1_{\{1\le|\xi|\le 1+R^{-1/2}\}}
e^{-ir_0|\xi|+it_0|\xi|^2},$$
where $r_0\in [R/2, R]$ and $t_0\in \R$. Thus the left-hand side of \eqref{eq:linear q=2} is
comparable to $$\left(\int_{R/2}^R \int_{\R}\left|\int_{1}^{1+R^{-1/2}}e^{[-i(t-t_0)s^2+i(\pm
r-r_0)s]}ds\right|^2 dt\,dr\right)^{1/2},$$ where by $\pm$ it denotes a summation of two terms. We
observe that
\begin{align*}&\left|\int_1^{1+R^{-1/2}}e^{[-i(t-t_0)s^2+i(\pm r-r_0)s]}ds\right|\\
&\quad \quad =\left|\int_1^{1+R^{-1/2}} e^{-i\{(t-t_0)(s-1)^2-[(\pm
r-r_0)-2(t-t_0)](s-1)\}}ds\right|,
\end{align*} and hence choose $r\in[R/2, R]$ and $t\in\R$ such that
\begin{align*} R/100\le t-t_0&\le R/50,\\
\left|(r-r_0)-2(t-t_0)](s-1)\right|&\le R^{1/2}/100.
\end{align*} Thus $r$ and $t$ are in the intersection region of two tubes whose size is of
 $R^{1/2}\times R$.

With this choice of $r$ and $t$, $\left|(t-t_0)(s-1)^2-[(r-r_0)-2(t-t_0)](s-1)\right|$ is less
than a small number, say $\pi/6$. Then by direct computations, the term with $+$ sign will be
bounded below by $R^{1/4}$. However for the term with $-$ sign, given this choice of $r$ and $t$,
we see that the roots of the quadratic polynomial, $(t-t_0)(s-1)^2-[(r+r_0)+2(t-t_0)](s-1)$, will
be strictly less than $-1$, which consequently are not located in the interval $[0,R^{1/2}]$. Thus
by the principle of non-stationary phase, we see that the term with $-$ sign will be bounded above
by $O_N(R^{-N})$ for any $N>0$. Then by choosing $N$ sufficiently large, from the triangle
inequality the left-hand side of \eqref{eq:linear q=2} $\gtrsim R^{1/4}$. Also it is easy to see
that its right-hand side $\lesssim R^{1/4}$. Thus we see that the estimate $L^2\to L^2$ when $R\ge
2$ is sharp.
\end{example}

Our second counterexample is to show that the estimates in the region $II$ in Figure
\ref{fig:lin-dya} determined by the lines $q=2$ and $q=4$ are sharp. In this example we will only
show the estimates on the line $q=2$ in Proposition \ref{prop: lin-q=2} are sharp by using the
principle of stationary phase.
\begin{example}[\textbf{II}]\label{ex:lin-II}
If $R\ge 2$, the estimate $L^p\to L^2$ when $2\le p\le \infty$ goes back to Proposition \ref{prop:
lin-q=2}. We take the example
$$f(|\xi|^2, \xi)=F(|\xi|)=|\xi|^{-\frac{n-2}{2}}1_{\{1\le |\xi|\le 2\}}
e^{-ir_0|\xi|+it_0|\xi|^2},$$ where $r_0\in [R/2, R]$ and $t_0\in \R$. Then the left-hand side of
\eqref{eq:linear q=2} is comparable to
$$\left(\int_{R/2}^R\int_{\R}\left|\int_I e^{-i\{(t-t_0)(s-1)^2-[(\pm
r-r_0)-2(t-t_0)](s-1)\}}ds\right|^2 dt\,dr\right)^{1/2}.$$ We choose $r\in [R/100, R/50]$ and
$t\in \R$ such that ${(r-r_0)}/{2(t-t_0)}\in [1,2]$. Then
$$-\frac{-[(r-r_0)-2(t-t_0)]}{2(t-t_0)}\in[0,1],\,-\frac{(r+r_0)+2(t-t_0)}{2(t-t_0)}<-1.$$
Then from the principles of stationary phase and non-stationary phase,
\begin{align*}
\left|\int_I e^{-i\{(t-t_0)(s-1)^2-[(r-r_0)-2(t-t_0)](s-1)\}}ds\right|&\gtrsim R^{-1/2},\\
\left|\int_I e^{-i\{(t-t_0)(s-1)^2+[( r+r_0)+2(t-t_0)](s-1)\}}ds\right|&\lesssim_N R^{-N},
\end{align*} for any $N\ge 0$. Then if choosing $N$ sufficiently large, the triangle inequality
gives
$$\left(\int_{R/2}^R\int_{\R}\left|\int_Ie^{-i\{(t-t_0)(s-1)^2-[(\pm r-r_0)-2(t-t_0)]
(s-1)\}}ds\right|^2 dt\,dr\right)^{1/2}\gtrsim R^{1/2}.$$ Its right-hand side $\lesssim R^{1/2}$
for $2\le p\le \infty$. Thus we see that the estimates on the line $q=2$ when $R\ge 2$ are sharp.
\end{example}

The third counterexample shows that the estimates inside the region $III$ determined by lines
$q=4$, $q=3p'$ and $q=\infty$ in Figure \ref{fig:lin-dya} are sharp. In this example, we will only
carry out the computations for the estimates $L^p\to L^q$ on the line $q=\infty$ in Proposition
\ref{prop:lin-q=infty}.

\begin{example}[\textbf{III}]\label{ex:lin-III}
If $R\ge 2$, the estimate $L^p\to L^{\infty}$ when $1\le p\le \infty$ goes back to Proposition
\ref{prop:lin-q=infty}. We take
$$f(|\xi|^2, \xi)=F(|\xi|)=|\xi|^{-\frac{n-2}{2}}1_{\{1\le |\xi|\le 2\}}e^{-ir_0|\xi|+it_0|\xi|^2},$$
where $r_0\in [R/2, R]$ and $t_0\in \R$. They are chosen such that
$\|(fd\sigma)^{\vee}\|_{L^{\infty}_{t,x}}$ can be realized at $(t_0, x_0)$ with $r_0=|x_0|$. Hence
the left-hand side of the inequality \eqref{eq:linear q=infty} is comparable to
$$R^{-\frac{n-2}{2}}\left|\int_{I}e^{i(\pm r_0-r_0)s} ds\right|.$$
Since $r_0\in [R/2, R]$, $\left|\int_{I}e^{-i2r_0s} ds\right|\lesssim R^{-N}$ holds for any $N\ge
0$. Then the triangle inequality yields $R^{-\frac{n-2}{2}}\left|\int_{I}e^{i(\pm r_0-r_0)s}
ds\right|\gtrsim R^{-\frac{n-2}{2}}$. On the other hand, the right-hand side of \eqref{eq:linear
q=infty}$\lesssim R^{-(n-2)/2}$ for $1\le p\le \infty$. Hence the estimates on the line $q=\infty$
when $R\ge 2$ are sharp.

For lines $q=4$ and $q=3p'$, the estimates go back to \eqref{eq:lin-q=4} and \eqref{eq:linear
q=3p'}. We choose $r\in [R/2, R]$ and $t$ such that $2\le r-r_0\le 4,\,2\le t-t_0\le 4$. Then by
the same reasoning as Example \ref{ex:lin-I}, these estimates are sharp.
\end{example}

Thus the proof of Theorem \ref{thm:dyadic-lin} is complete.

\section{Proof of Theorem \ref{thm:dyadic-bilin}: bilinear estimates and examples}\label{sec:biliear}
For $f\in \L_1$ and $g\in \L_M$ with $0<M\le 1/4$, we set $I_1=[1, 2],\,I_M=[M, 2M]$ and
$F(|\xi|)=f(|\xi|^2, \xi), G(|\xi|)=g(|\xi|^2, \xi)$. In the bilinear case, $1$ and $1/M$ will
bring in two natural separation scales in the physical space. In light of the proof of Theorem
\ref{thm:dyadic-lin}, we will have the following permutations of the product
$(fd\sigma_1)^{\vee}(gd\sigma_2)^{\vee}$.
\begin{itemize}
\item\label{case 1} when $|x|=r\ge 1/M$,
\begin{align*}
|(fd\sigma_1)^{\vee}(gd\sigma_2)^{\vee}|=|\M f\M g+\M f\E g+\M g\E f+\E f\E g|.
\end{align*}
\item\label{case 2} when $1\le |x|=r\le 1/M$,
\begin{equation*}
|(fd\sigma_1)^{\vee}(gd\sigma_2)^{\vee}|=|\M f(gd\sigma_2)^{\vee}+\E f(gd\sigma_2)^{\vee}|.
\end{equation*}
\item\label{case 3}  when $|x|\le 1$, $|(fd\sigma_1)^{\vee}(gd\sigma_2)^{\vee}|$ remains
unchanged.
\end{itemize}
We are going to prove the ``estimates" part of Theorem \ref{thm:dyadic-bilin} via the following
three propositions and its ``sharpness" part by building counterexamples in three cases followed.
\begin{proposition}\label{prop:sml-bilin-err}
Suppose $f\in \L_1$ and $g\in \L_M$ with $0<M\le 1/4$, and $R\le 1$ is a dyadic number. Then we
have sharp estimates
\begin{itemize}
\item for $q=1$ and $2\le p\le \infty$,
\begin{equation}\label{ex:loc-22}
\|(fd\sigma_1)^{\vee}(gd\sigma_2)^{\vee}\|_{L^1_{t,x}}\lesssim
R^{n-1}M^{-1+\frac{n-1}{p'}}\|f\|_{L^p(S_1)}\|g\|_{L^p(S_2)}.
\end{equation}
\item for $q\ge \max\{2, p'\}$,
\begin{equation}\label{ex:loc-21}
\|(fd\sigma_1)^{\vee}(gd\sigma_2)^{\vee}\|_{L^q_{t,x}}\lesssim
R^{\frac{n-1}{q}}M^{\frac{n-1}{p'}}\|f\|_{L^p(S_1)}\|g\|_{L^p(S_2)}.
\end{equation}
\end{itemize}
\end{proposition}

\begin{proof}
If we change to the polar coordinates, the left-hand side of \eqref{ex:loc-22} reduces to
\begin{align}\label{eq:loc-32}
\int_{R/2}^R\int_{\R}&\left|\int_{I_1}F(s_1)e^{-its_1^2}(d\mu)^{\vee}(rs_1e_1)s_1^{n-2}ds_1
\times \right.\nonumber\\
&\left.\times\int_{I_M}G(s_2)e^{-its_2^2}(d\mu)^{\vee}(rs_2e_1)s_2^{n-2}ds_2\right|dtr^{n-2}dr.
\end{align} We use the Cauchy-Schwarz inequality
and the Plancherel theorem in $t$ to bound \eqref{eq:loc-32} by
\begin{equation}\label{eq:loc-31}
R^{n-1}M^{n-2}\|F\|_{L^2(I_1)}M^{-1/2}\|G\|_{L^{2}(I_M)}.
\end{equation} Then by the H\"older inequality,
\eqref{eq:loc-31} is bounded by $ R^{n-1}M^{n-2-\frac{n-1}{p}}\|f\|_{L^p(S_1)}\|g\|_{L^p(S_2)}. $
Hence the inequality \eqref{ex:loc-22} follows.

To prove \eqref{ex:loc-21}, by the H\"older inequality,
\begin{equation*}
\|(fd\sigma_1)^{\vee}(gd\sigma_2)^{\vee}\|_{L^q_{t,x}}\lesssim \|(fd\sigma_1)^{\vee}\|_{L^q_{t,x}}
\|(gd\sigma_2)^{\vee}\|_{L^{\infty}_{t,x}}.
\end{equation*} On the one hand, by Proposition \ref{prop: sml-err},
$\|(fd\sigma_1)^{\vee}\|_{L^q_{t,x}}\lesssim R^{\frac{n-1}{q}}\|f\|_{L^p(S_1)}. $ On the other
hand, by the H\"older inequality, $\|(gd\sigma_2)^{\vee}\|_{L^{\infty}_{t,x}}\lesssim
M^{\frac{n-1}{p'}}\|g\|_{L^p(S_2)}$. Hence the inequality \eqref{ex:loc-21} follows.
\end{proof}

The following proposition concerns the case where $1\le|x|\le 1/M$.
\begin{proposition}\label{prop: med-bilin-error}
Suppose $f\in \L_1$ and $g\in \L_M$ with $0<M\le 1/4$, and $2\le R\le 1/M$. Then
\begin{itemize}
\item for $q=1$ and $2\le p\le \infty$,
\begin{equation}\label{eq:loc-29}
\|(fd\sigma_1)^{\vee}(gd\sigma_2)^{\vee}\|_{L^1_{t,x}}\lesssim R^{\frac n2}
M^{-1+\frac{n-1}{p'}}\|f\|_{L^p(S_1)}\|g\|_{L^p(S_2)}.
\end{equation}
\item for $q\ge \max\{2, p'\}$,
\begin{equation}\label{eq: loc-15}
\|(fd\sigma_1)^{\vee}(gd\sigma_2)^{\vee}\|_{L^q_{t,x}}\lesssim \|R^*\|_{L^p\to
L^q}M^{\frac{n-1}{p'}}\|f\|_{L^p(S_1)}\|g\|_{L^p(S_2)},
\end{equation}
where $\|R^*\|_{L^p\to L^q}$ denotes the operator norm of $f\to (fd\sigma)^{\vee}$ from $L^p(S_1)$
to $L^{q}_{t,x}$ given by Theorem \ref{thm:dyadic-lin}.
\end{itemize}
\end{proposition}

\begin{proof}
To prove \eqref{eq:loc-29}, it suffices to prove the following inequalities by Lemma
\ref{lem:Four-Bess},
\begin{equation*}
\int_{R/2}^{R}\int_{\R}|\M f(t, re_1)||(gd\sigma_2)^{\vee}(t,re_1)|dt\,r^{n-2}dr\lesssim R^{\frac
n2}M^{-1+\frac{n-1}{p'}}\|f\|_{L^p(S_1)}\|g\|_{L^p(S_2)},
\end{equation*}
and
\begin{equation*}
\int_{R/2}^{R}\int_{\R}|\E f(t,re_1)||(gd\sigma_2)^{\vee}(t,re_1)|dt\,r^{n-2}dr \lesssim
R^{\frac{n}{2}-1}M^{-1+\frac{n-1}{p'}}\|f\|_{L^p(S_1)}\|g\|_{L^p(S_2)}.
\end{equation*}
These two estimates above can be proven along similar lines as proving \eqref{ex:loc-22}. We
choose to prove the second. The H\"older inequality yields,
\begin{equation*}
\|(fd\sigma_1)^{\vee}(gd\sigma_2)^{\vee}\|_{L^q_{t,x}}\lesssim
\|(fd\sigma_1)^{\vee}\|_{L^q_{t,x}}\|(gd\sigma_2)^{\vee}\|_{L^{\infty}_{t,x}}.
\end{equation*} Then by using the same reasoning as in the proof of \eqref{ex:loc-21},
the inequality \eqref{eq: loc-15} follows. We note that the error term $\E f$ gives a better decay
estimate as expected.
\end{proof}

Next let us concentrate on the case where $|x|\ge 1/M$. As indicated at the beginning of this
section, we will have to deal with estimates involving $|\M f\M g|$, $|\M f\E g|$, $|\E f \M g|$
and $|\E f\E g|$.

\begin{proposition}[Bilinear main term estimates]\label{prop: l-bilin-error}
Suppose $f\in \L_1$ and $g\in \L_M$ with $0<M\le 1/4$, and $R\ge 1/M$. Then
\begin{itemize}
\item for $q=1$ and $2\le p\le \infty$,
\begin{equation}\label{eq: loc-11}
\|(fd\sigma_1)^{\vee}(gd\sigma_2)^{\vee}\|_{L^1_{t,x}}\lesssim
RM^{\frac{n-2}{2}-\frac{n-1}{p}}\|f\|_{L^p(S_1)}\|g\|_{L^p(S_2)}.
\end{equation}
\item for $q=2$ and $2\le p\le \infty$,
\begin{equation}\label{eq: loc-12}
\|(fd\sigma_1)^{\vee}(gd\sigma_2)^{\vee}\|_{L^2_{t,x}}\lesssim R^{-\frac{n-2}{2}}
M^{\frac{n-1}{2}-\frac{n-1}{p}}\|f\|_{L^p(S_1)}\|g\|_{L^p(S_2)}.
\end{equation}
\item for $q\ge \max\{4, 3p'\}$ and $1\le p\le \infty$,
\begin{equation}\label{eq:loc-4}
\|(fd\sigma_1)^{\vee}(gd\sigma_2)^{\vee}\|_{L^q_{t,x}}\lesssim \|R^*\|_{L^p\to
L^q}R^{-\frac{n-2}{2}}M^{\frac{n}{2}-\frac{n-1}{p}}\|f\|_{L^p(S_1)}\|g\|_{L^p(S_2)}.
\end{equation}
\end{itemize}
\end{proposition}

\begin{proof}
To prove \eqref{eq: loc-11}, it suffices to prove the following inequalities
\begin{align}\label{eq: loc-10}
&\|\M f\M g\|_{L^1_{t,x}}\lesssim RM^{\frac{n-2}{2}-\frac{n-1}{p}}\|f\|_{L^p(S_1)}
\|g\|_{L^p(S_2)}.\\
&\|\E f\E g\|_{L^1_{t,x}}\lesssim R^{-1}
M^{\frac{n-4}{2}-\frac{n-1}{p}}\|f\|_{L^p(S_1)}\|g\|_{L^p(S_2)}.\nonumber\\
&\|\M f\E g\|_{L^1_{t,x}}\lesssim
M^{\frac{n-4}{2}-\frac{n-1}{p}}\|f\|_{L^p(S_1)}\|g\|_{L^p(S_2)}.\nonumber\\
& \|\E f\M g\|_{L^1_{t,x}}\lesssim
M^{\frac{n-2}{2}-\frac{n-1}{p}}\|f\|_{L^p(S_1)}\|g\|_{L^p(S_2)}.\nonumber
\end{align}

In what follows we will only prove \eqref{eq: loc-10} since other estimates involving error terms
will follow similarly. In fact these inequalities give better decay estimates than those given by
\eqref{eq: loc-10}. By the heuristic approximation \eqref{prop: heur-four-bess} with $+$ sign, the
left-hand side of \eqref{eq: loc-10} reduces to
\begin{equation}\label{eq: loc-14}
\int_{R/2}^R\int_{\R}\left|\int_{I_1}F(s_1)s_1^{\frac{n-2}{2}}e^{i(rs_1-ts_1^2)}ds_1
\int_{I_M}G(s_2)s_2^{\frac{n-2}{2}}e^{i(rs_2-ts_2^2)}\right|dt\,dr.
\end{equation}
After changing variables and using the Cauchy-Schwarz inequality and the Plancherel theorem in
$t$, we see that \eqref{eq: loc-14} is bounded by
\begin{equation*}
RM^{\frac{n-3}{2}}\|F\|_{L^2(I_1)}\|G\|_{L^2(I_M)}\sim
RM^{-\frac{1}{2}}\|f\|_{L^2(S_1)}\|g\|_{L^2(S_2)}.
\end{equation*}
Then from the H\"older inequality, the inequality \eqref{eq: loc-10} follows. Similarly, to prove
\eqref{eq: loc-12}, it suffices to prove the following inequalities
\begin{align}\label{eq: loc-9}
&\|\M f\M g\|_{L^2_{t,x}}\lesssim R^{-\frac{n-2}{2}}
M^{\frac{n-1}{2}-\frac{n-1}{p}}\|f\|_{L^p(S_1)}\|g\|_{L^p(S_2)}.\\
&\|\E f\E g\|_{L^2_{t,x}}\lesssim R^{-\frac{n+1}{2}}
M^{\frac{n-2}{2}-\frac{n-1}{p}}\|f\|_{L^p(S_1)}\|g\|_{L^p(S_2)}.\nonumber\\
&\|\M f\E g\|_{L^2_{t,x}}\lesssim
R^{-\frac n2+\frac12}M^{\frac{n-2}{2}-\frac{n-1}{p}}\|f\|_{L^p(S_1)}\|g\|_{L^p(S_2)}.\nonumber\\
& \|\E f\M g\|_{L^2_{t,x}}\lesssim
R^{-\frac{n}{2}+\frac12}M^{\frac{n}{2}-\frac{n-1}{p}}\|f\|_{L^p(S_1)}\|g\|_{L^p(S_2)}.\nonumber
\end{align}
Since the estimates above involving error terms give better decay estimates than \eqref{eq:
loc-9}, we will also only prove \eqref{eq: loc-9}. We rewrite its left-hand side as
\begin{align*}
R^{-\frac{n-2}{2}}&\left(\int_{R/2}^R\int_{\R}\left|\int_{I_1\times I_M}F(s_1)G(s_2)\times
\right.\right.\\
&\left.\left.\quad\times(s_1s_2)^{\frac{n-2}{2}}e^{i(r(s_1+s_2)-t(s_1^2+s_2^2)}ds_1ds_2\right|^2dt
dr\right)^{1/2}.
\end{align*}
Setting $x:=s_1+s_2$ and $y:=s_1^2+s_2^2$, we observe that the Jacobian $\sim |1-M|\sim 1$
provided $M\le 1/4$. From the Plancherel theorem both in $t$ and $r$, the left-hand side of
\eqref{eq: loc-9} is further majorized by
\begin{equation*}
R^{-\frac{n-2}{2}}M^{\frac{n-2}{2}}\|F\|_{L^2(I_1)}\|G\|_{L^2(I_M)}\sim
R^{-\frac{n-2}{2}}\|f\|_{L^2(S_1)}\|g\|_{L^2(S_2)}.
\end{equation*} By using the H\"older inequality again, we see that \eqref{eq: loc-9} follows.
Finally we prove \eqref{eq:loc-4}. In fact, it suffices to prove the following two inequalities
\begin{align*}
\|(fd\sigma_1)^{\vee}\M g\|_{L^q_{t,x}}\lesssim \|R^*\|_{L^p\to L^q}R^{-\frac{n-2}{2}}
M^{\frac{n}{2}-\frac{n-1}{p}}\|f\|_{L^p(S_1)}\|g\|_{L^p(S_2)}.\\
\|(fd\sigma_1)^{\vee}\E g\|_{L^q_{t,x}}\lesssim \|R^*\|_{L^p\to
L^q}R^{-\frac{n}{2}}M^{\frac{n}{2}-1-\frac{n-1}{p}}\|f\|_{L^p(S_1)}\|g\|_{L^p(S_2)}.
\end{align*} The first follows from the H\"older inequality and the linear estimate in Theorem
\ref{thm:dyadic-lin}, and the second follows along similar lines.
\end{proof}

Therefore the restriction estimates in Theorem \ref{thm:dyadic-bilin} are obtained from
Propositions \ref{prop:sml-bilin-err}, \ref{prop: med-bilin-error} and \ref{prop: l-bilin-error}.

In the remainder of this section we will construct counterexamples to show these estimates are
sharp or nearly sharp up to $R^{\eps}$. Since the error terms give much better decay estimates, we
will use the heuristic approximations \eqref{prop: heur-four-bess} of $(fd\sigma_1)^{\vee}$ and
$(gd\sigma_2)^{\vee}$ when computing these examples. We will distinguish them into three cases as
follows.

\underline{\textit{Case 1}: $R\ge 1/M$.}

We start with a common example to show the estimates in the region $I$ in Figure \ref{fig:
bilin-dya} determined by $L^2\times L^2\to L^1$, $L^2\times L^2\to L^2$ and $L^1\times L^1\to
L^{\infty}$ are sharp by using the idea coming from standard Knapp example. In the following
example, we will only do the computations when $p=2$ and $q=1$.

\begin{example}[\textbf{I}]\label{ex:I-large}
If $R\ge 1/M$, $R M^{-\frac12}$ is best possible in the following inequality
\begin{equation}\label{eq: loc-8}
\|(fd\sigma_1)^{\vee} (gd\sigma_2)^{\vee}\|_{L^1_{t,x}} \lesssim R
M^{-\frac12}\|f\|_{L^2(S_1)}\|g\|_{L^2(S_2)}.
\end{equation}

We take \begin{align*}
f(|\xi|^2,\xi)&=F(|\xi|)=|\xi|^{-\frac{n-2}{2}}e^{i(-r_0|\xi|+t_0|\xi|^2)}1_{\{1\le
|\xi|\le 1+R^{-1}M\}},\\
g(|\eta|^2,\eta)&=G(|\eta|)=|\eta|^{-\frac{n-2}{2}}e^{i(-r_0|\eta|+t_0|\eta|^2)}1_{\{M\le
|\eta|\le M+R^{-1}\}},\end{align*} where $r_0\in [R/2, R]$ and $t_0\in \R$. By the heuristic
approximation \eqref{prop: heur-four-bess}, the left-hand side of \eqref{eq: loc-8} is comparable
to
\begin{equation*}
\int_{R/2}^{R} \int_{\R} \left|\int_1^{1+R^{-1}M} e^{i[(\pm r-r_0)s_1-(t-t_0)s_1^2]}
ds_1\int_{M}^{M+R^{-1}} e^{i[(\pm r-r_0)s_2-(t-t_0)s_2^2]}ds_2\right|dt\,dr,
\end{equation*} which we understood is a summation of four terms.
For the integral on $[1, 1+R^{-1}M]$, we will choose $r$ and $t$ such that $R/100\le r-r_0\le
R/50$ and $RM^{-1}/100\le t-t_0\le RM^{-1}/50$. Then we have
\begin{align*}
|[(r-r_0)-2(t-t_0)](s_1-1)-(t-t_0)(s_1-1)^2|&\le 1, \\
-\frac{(r+r_0)+2(t-t_0)}{2(t-t_0)}=-\left(1+\frac{r-r_0}{2(t-t_0)}\right)&\notin[0,R^{-1}M].
\end{align*}
Hence the integral on $[1, 1+R^{-1}M]$ with $+$ sign is $\gtrsim R^{-1}M$ while the one with $-$
sign $\lesssim_N R^{-N}$ for any $N\ge 0$. Similarly for the integral on $[M, M+R^{-1}]$ with this
choice of $r$ and $t$. Then if $N$ is sufficiently large, the triangle inequality gives
\begin{align*}
\int_{R/2}^{R}&\int_{\R} \left|\int_1^{1+R^{-1}M} e^{i[(\pm r-r_0)s_1-(t-t_0)s_1^2]}ds_1\right.\\
&\left.\times \int_{M-R^{-1}}^M e^{i[(\pm r-r_0)s_2-(t-t_0)s_2^2]} ds_2\right|dt\,dr\gtrsim 1.
\end{align*}
Then by direct computations, the right-hand side of \eqref{eq: loc-8} $\lesssim 1$. Thus we see
that the estimate $L^2\times L^2\to L^1$ is sharp when $R\ge 1/M$.
\end{example}

By modifying the above ``narrow'' Example \ref{ex:I-large}, namely taking a linear combination to
create a ``spreading-out'' example, we will show that the estimates in the region $II$ in Figure
\ref{fig: bilin-dya} determined by the lines $q=1$ and $q=2$ are sharp by using the Khintchine
inequality. A similar construction by Lee and Vargas can be found in
\cite{Lee-Vargas:2008:null-form-wave} to show the sharp null form estimates for the wave equation.
In the following example we will only do computations for the estimates on the line $q=1$.

\begin{example}[\textbf{II}]\label{ex:II-larger}
If $R\ge 1/M$, $RM^{(n-2)/2-(n-1)/p}$ is best possible in the following inequality
\begin{equation}\label{eq:loc-27}
\|(fd\sigma_1)^{\vee} (gd\sigma_2)^{\vee}\|_{L^1_{t,x}}\lesssim
RM^{\frac{n-2}{2}-\frac{n-1}{p}}\|f\|_{L^p(S_1)}\|g\|_{L^p(S_2)},
\end{equation} where $2\le p\le \infty$.

We define two index sets $J:=\{j\in \Z: 1\le j\le [RM^{-1}]\}$ and $K:=\{k\in \Z: 1\le k\le
[RM]\}$, where $[x]$ denotes the biggest integer which is less than or equal to $x\in\R$. For each
$j\in J, k\in K$ we define
\begin{align*} &f_j(|\xi|^2,\xi)=F_j(|\xi|)=|\xi|^{\frac{n-2}{2}}e^{i(-r_0|\xi|+t_0|\xi|^2)}
1_{\{1+(j-1)R^{-1}M\le |\xi|\le 1+jR^{-1}M\}},\\
& g_k(|\eta|^2,\eta)=G_k(|\eta|)=|\eta|^{-\frac{n-2}{2}}e^{i(-r_0|\eta|+t_0|\eta|^2)}
1_{\{M+(k-1)R^{-1}\le |\eta|\le M+kR^{-1}\}}.
\end{align*}

Also we set $f=\sum_{j\in J}\eps_j f_j$ and $g=\sum_{k\in K}\tilde{\eps}_k g_k$, where
$\{\eps_j:j\in J\}$ and $\{\tilde{\eps}_k: k\in K\}$ are sets of i.i.d. ~(independent identically
distributed) random variables taking $\pm1$ with an equal probability $1/2$. Note that $f_j$ and
$g_k$ are ``narrow'', disjoint and in the form of Example \ref{ex:I-large}; but $f$ and $g$
``spread out'' and support on the whole set $S_1$ and $S_2$. By the Khintchine inequality, we
estimate the left-hand side of \eqref{eq:loc-27} by
\begin{equation}\label{eq:loc-6}
\mathbb{E}\left(\|(fd\sigma_1)^{\vee} (gd\sigma_2)^{\vee}\|_{L^1_{t,x}}\right)\sim
\|(\sum_{j,k}|(f_jd\sigma_1)^{\vee}(g_kd\sigma_2)^{\vee}|^2)^{1/2}\|_{L^1_{t,x}},
\end{equation} where $\mathbb{E}(X)$ denotes the expectation of the random variable $X$.
By the heuristic approximation \eqref{prop: heur-four-bess}, the right-hand side of
\eqref{eq:loc-6} is comparable to
\begin{align}\label{eq:loc-7}
\int_{R/2}^{R} \int_{\R} &\left(\sum_{j,k}\left|\int_{a_j-R^{-1}M}^{a_j}
F_j(s_1)e^{i[(\pm r-r_0)s_1-(t-t_0)s_1^2]}ds_1\times \right.\right.\nonumber\\
&\left.\left.\times \int_{b_k-R^{-1}}^{b_k} G_k(s_2)e^{i[(\pm
r-r_0)s_2-(t-t_0)s_2^2]}ds_2\right|^2\right)^{1/2}dt\,dr,
\end{align}
where $a_j=1+j R^{-1}M$ and $b_k=M+kR^{-1}$, and by $\pm$ we denote a summation of four terms. We
choose $r$ and $t$ such that $R/100\le r-r_0\le R/50$ and $RM^{-1}\le t-t_0\le RM^{-1}/50$. By
this choice of $r$ and $t$ and similar discussions as in Example \ref{ex:I-large}, the triangle
inequality gives,
\begin{align*}
|\int_{a_j-R^{-1}M}^{a_j}&F_j(s_1)e^{i[(\pm r-r_0)s_1-(t-t_0)s_1^2]}ds_1\times\\
&\times \int_{b_k-R^{-1}}^{b_k} G_k(s_2)e^{i[(\pm r-r_0)s_2-(t-t_0)s_2^2]}ds_2|
 \gtrsim R^{-2}M.
\end{align*} Then \eqref{eq:loc-7} is bounded below by $R^2M^{-1}(|J||K|)^{1/2} R^{-2}M$, i.e.,
$(|J||K|)^{1/2}$. Here $|J|\sim RM^{-1}$ denotes the cardinality of the index set $J$, similarly
for $|K|\sim MR$. Hence we obtain that the left-hand side of \eqref{eq:loc-27} is $\gtrsim R$. On
the other hand, the right-hand side of \eqref{eq:loc-27} $\lesssim R$ for $2\le p\le \infty$.
Hence the estimates on the line $q=1$ when $R\ge 1/M$ are sharp.
\end{example}

The following example shows that the estimates in the region $III$ in Figure \ref{fig: bilin-dya}
determined by the lines $q=2$ and $q=4$ are sharp by the principle of stationary phase. We do
computations when $q=2$.
\begin{example}[\textbf{III}]\label{ex:III-large}
If $R\ge 1/M$, $R^{-\frac{n-2}{2}}M^{\frac{n-1}{2}-\frac{n-1}{p}}$ is best possible in the
following inequality.
\begin{equation}\label{eq:loc-26}
\|(fd\sigma_1)^{\vee} (gd\sigma_2)^{\vee}\|_{L^2_{t,x}}\lesssim
R^{-\frac{n-2}{2}}M^{\frac{n-1}{2}-\frac{n-1}{p}}\|f\|_{L^p(S_1)}\|g\|_{L^p(S_2)},
\end{equation} where $2\le p\le \infty$.

We take
\begin{align*}
f(|\xi|^2,\xi)&=|\xi|^{-\frac{n-2}{2}}e^{i(-r_0|\xi|+t_0|\xi|^2)}1_{\{1\le |\xi|\le 2\}},\\
g(|\eta|^2,\eta)&=|\eta|^{-\frac{n-2}{2}}e^{i(-r_0|\eta|+t_0|\eta|^2)}1_{\{M\le|\eta|\le2M\}},
\end{align*} where $r_0\in [R/2, R]$ and $t_0\in \R$. By the heuristic approximation
\eqref{prop: heur-four-bess}, the left-hand side of \eqref{eq:loc-26} is comparable to
\begin{align*}
R^{-\frac{n-2}{2}}\left(\int_{R/2}^{R}\int_{\R} \left|\int_{I_1}e^{i[(\pm
r-r_0)s_1-(t-t_0)s_1^2]}ds_1\int_{I_M}e^{i[(\pm
r-r_0)s_2-(t-t_0)s_2^2]}ds_2\right|^2dt\,dr\right)^{1/2}.
\end{align*}

We choose $r\in [R/2, R]$ and $t$ such that $M^{-1}/100\le r-r_0\le M^{-1}/50$ and $
{(r-r_0)}/{2(t-t_0)}\in I_1$. Then from the principles of stationary phase and non-stationary
phase, for any $N\ge 0$,
\begin{align*}
\left|\int_{I_1}e^{i[(r-r_0)s_1-(t-t_0)s_1^2]}ds_1\right|&\gtrsim M^{1/2},\\
\left|\int_{I_1}e^{-i[(r+r_0)s_1+(t-t_0)s_1^2]}ds_1\right|&\lesssim_N M^{N},\\
\left|\int_{s_2\sim M}e^{i[(r-r_0)s_2-(t-t_0)s_2^2]}ds_2\right|&\gtrsim M,\\
\left|\int_{s_2\sim M}e^{-i[(r+r_0)s_2+(t-t_0)s_2^2]}ds_2\right|&\lesssim_N M^{N}.
\end{align*}
Then from the triangle inequality, the left-hand side of \eqref{eq:loc-26} $\gtrsim
R^{-(n-2)/2}M^{1/2}$. By direct computations, the right-hand side of \eqref{eq:loc-26} $\lesssim
R^{-(n-2)/2}M^{1/2}$. Thus the estimates on the line $q=2$ when $R\ge 1/M$ are sharp.
\end{example}

The next example will show that the estimates in the region $IV$ in Figure \ref{fig: bilin-dya}
determined by $L^2\times L^2\to L^2$, $L^4\times L^4\to L^4$ and $L^1\times L^1\to L^{\infty}$ are
sharp up to $R^{\eps}$ by using the idea of Knapp example. We will only do computations for the
estimate when $p=q=2$.

\begin{example}[\textbf{IV}]\label{ex:IV-large}
If $R\ge 1/M$, $R^{-\frac{n-2}{2}}$ is best possible in the following inequality.
\begin{equation}\label{eq:loc-25}
\|(fd\sigma_1)^{\vee} (gd\sigma_2)^{\vee}\|_{L^2_{t,x}}\lesssim
R^{-\frac{n-2}{2}}\|f\|_{L^2(S_1)}\|g\|_{L^2(S_2)}.
\end{equation}

We take \begin{align*}
f(|\xi|^2,\xi)&=|\xi|^{-\frac{n-2}{2}}e^{i(-r_0|\xi|+t_0|\xi|^2)}1_{\{1\le |\xi|\le 1+M^{1/2}\}},\\
g(|\eta|^2, \eta)&=|\eta|^{-\frac{n-2}{2}}e^{i(-r_0|\eta|+t_0|\eta|^2)}1_{\{M\le|\eta|\le 2M\}},
\end{align*} where $r_0\in [R/2, R]$ and $t_0\in \R$.
By the heuristic approximation \eqref{prop: heur-four-bess}, the left-hand side of
\eqref{eq:loc-25} is comparable to
\begin{align*}
R^{-\frac{n-2}{2}}\left(\int_{R/2}^{R}\int_{\R} \left|\int_{1}^{1+M^{\frac12}}e^{i[(\pm
r-r_0)s_1-(t-t_0)s_1^2]}ds_1\int_{I_M}e^{i[(\pm
r-r_0)s_2-(t-t_0)s_2^2]}ds_2\right|^2dt\,dr\right)^{1/2}.
\end{align*}

We choose $r\in [R/2, R]$ and $t\in \R$ such that
\begin{align*}
|(r-r_0)-2(t-t_0)|&\le M^{-1/2},\\ |(r-r_0)-2M(t-t_0)|&\le M^{-1},\\
 M^{-1}/100\le t-t_0&\le M^{-1}/50,
\end{align*}
i.e., $r$ and $t$ are located in the intersection area of two tubes which has size $M^{-1}\times
M^{-1/2}$.  Then by similar discussions as Example \ref{ex:I-large}, the left-hand side of
\eqref{eq:loc-25} $\gtrsim R^{-(n-2)/2}M^{3/4}$; by direct computations the right-hand side of
\eqref{eq:loc-25} $\lesssim R^{-(n-2)/2}M^{3/4}$. Hence we see that the estimate $L^2\times L^2\to
L^2$ when $R\ge 1/M$ is sharp.
\end{example}

The next example shows that the estimates in the region $V$ in Figure \ref{fig: bilin-dya}
determined by the lines $q=4$, $q=\infty$ and $q=3p'$ are sharp. In this example, we will do the
computations for the estimates on the line $q=\infty$.

\begin{example}[\textbf{V}]\label{ex:V-large}
If $R\ge 1/M$, $R^{-(n-2)}\times M^{\frac{n}{2}-\frac{n-1}{p}}$ is best possible in the following
inequality
\begin{equation}\label{eq:loc-5}
\|(fd\sigma_1)^{\vee}(gd\sigma_2)^{\vee}\|_{L^{\infty}_{t,x}} \lesssim
R^{-(n-2)}M^{\frac{n}{2}-\frac{n-1}{p}}\|f\|_{L^p(S_1)}\|g\|_{L^p(S_2)},
\end{equation} where $1\le p\le \infty$.

We take\begin{align*}
f(|\xi|^2, \xi)&=|\xi|^{-\frac{n-2}{2}}e^{i(-r_0|\xi|+t_0|\xi|^2)}1_{\{1\le|\xi|\le 2\}},\\
g(|\eta|^2, \eta)&=|\eta|^{-\frac{n-2}{2}}e^{i(-r_0|\eta|+t_0|\eta|^2)}1_{\{M\le|\eta|\le
2M\}},\end{align*} where $r_0\in [R/2, R]$ and $t_0\in \R$ satisfying the $L^{\infty}$ norms
$\|(fd\sigma)^{\vee}\|_{L^{\infty}_{t,x}(\R\times A_R)}$ and
$\|(gd\sigma)^{\vee}\|_{L^{\infty}_{t,x}(\R\times A_R)}$ can be realized at $(t_0,x_0)$ with
$|x_0|=r_0$. By the heuristic approximation \eqref{prop: heur-four-bess}, the left-hand side of
\eqref{eq:loc-5} is comparable to
\begin{equation*}
R^{-(n-2)}\left|\int_{I_1}e^{i(\pm r_0-r_0)s_1}ds_1\int_{I_M}e^{i(\pm r_0-r_0)s_2}ds_2\right|.
\end{equation*}
Then by the same reasoning as Example \ref{ex:lin-III}, the above $\gtrsim R^{-(n-2)}M$. On the
other hand, the right-hand side of \eqref{eq:loc-5} $\lesssim R^{-(n-2)}M$. Hence the estimates on
the line $q=\infty$ when $R\ge 1/M$ are sharp.

For lines $q=4, 4\le p\le \infty$ or $q=3p',1\le p<4$, the estimates go back to \eqref{eq:loc-4}.
We will choose $r\in [R/2, R]$ and $t$ such that $2\le r-r_0\le 4$ and $2\le t-t_0\le 4$. Then by
similar reasoning, the estimates on these lines are sharp.
\end{example}

\underline{\textit{Case 2}: $2\le R\le 1/M $.}

In this subcase, we will construct counterexamples to show the restriction estimates in Theorem
\ref{thm:dyadic-bilin} are sharp when $2\le R\le 1/M$. As in the \textit{Case 1}, we will start
with a ``narrow'' example which shows that estimates in the region $I$ in Figure \ref{fig:
bilin-dya} determined by $L^2\times L^2\to L^1$, $L^2\times L^2\to L^2$ and $L^1\times L^1\to
L^{\infty}$ are sharp. In this example, we will do computations for the estimate $L^2\times L^2\to
L^1$.

\begin{example}[\textbf{I}]\label{ex:I-med}
If $2\le R\le 1/M$, $R^{\frac n2}M^{\frac {n-3}{2}}$ is best possible in
\begin{equation}\label{eq:loc-19}
\|(fd\sigma_1)^{\vee} (gd\sigma_2)^{\vee}\|_{L^1_{t,x}} \lesssim R^{\frac n2}M^{\frac
{n-3}{2}}\|f\|_{L^2(S_1)}\|g\|_{L^2(S_2)}.
\end{equation}

We take \begin{align*} f(|\xi|^2, \xi)&=|\xi|^{-\frac{n-2}{2}}e^{i(-r_0|\xi|+t_0|\xi|^2)}1_{\{1\le
|\xi|\le 1+M^2\}},\\
g(|\eta|^2,\eta)&=|\eta|^{-(n-2)}e^{it_0|\eta|^2}1_{I_M},
\end{align*} where $r_0\in [R/2, R]$ and $t_0\in \R$. By the heuristic approximation
\eqref{prop: heur-four-bess} only for $(fd\sigma)^{\vee}$, we see that the left-hand side of
\eqref{eq:loc-19} is comparable to
\begin{align*}
R^{\frac{n-2}{2}}&\int_{R/2}^{R}\int_{\R} \left|\int_{1}^{1+M^2} e^{i((\pm
r-r_0)s_1-(t-t_0)s_1^2)}ds_1\int_{I_M} e^{-i(t-t_0)s_2^2} (d\mu)^{\vee}(rse_1)ds_2 \right|dt\,dr.
\end{align*}
We choose $r\in [R/2, R]$ and $t\in \R$ such that $$R/100\le r-r_0\le R/50,\,M^{-2}/100\le
t-t_0\le M^{-2}/50.$$ Then we have $|[(r-r_0)-2(t-t_0)](s_1-1)|\le c$ with a small $c>0$ and
$-\frac{(r+r_0)+2(t-t_0)}{2(t-t_0)}<-1$. From the principle of non-stationary phase and the
triangle inequality, the left-hand side of \eqref{eq:loc-19} $\gtrsim R^{\frac{n}{2}}M$. On the
other hand, the right-hand side of \eqref{eq:loc-19} $\lesssim R^{\frac{n}{2}}M$. Thus the
estimate $L^2\times L^2\to L^1$ when $2\le R\le 1/M$ is sharp.
\end{example}

In the next example, we are going to show the estimates in the region $II$ in Figure \ref{fig:
bilin-dya} determined by the lines $q=1$ and $q=2$ are sharp. In this example, we will do
computations for the estimates on the line $q=1$.

\begin{example}[\textbf{II}]\label{ex: II-med}
If $2\le R\le 1/M$, $R^{\frac n2}M^{-1+\frac {n-1}{p'}}$ is best possible in the following
inequality
\begin{equation}\label{eq:q=1, 2, 1<R<1/M}
\|(fd\sigma_1)^{\vee} (gd\sigma_2)^{\vee}\|_{L^1_{t,x}} \lesssim R^{\frac n2}M^{-1+\frac
{n-1}{p'}}\|f\|_{L^p(S_1)}\|g\|_{L^p(S_2)},
\end{equation} where $2\le p\le \infty$. We define an index set $J:=\{j: 1\le j\le [M^{-2}]\}$.
For each $j\in J$, we set
\begin{equation*}
 f_j(|\xi|^2,\xi)=F_j(|\xi|)=|\xi|^{-\frac{n-2}{2}}e^{i(-r_0|\xi|+t_0|\xi|^2)}1_{\{1+(j-1)M^{2}\le
|\xi|\le 1+jM^{2}\}}.
\end{equation*} Then we define $$f=\sum_{j}\eps_j f_j, \,\,
g(|\eta|^2,\eta)=|\eta|^{-(n-2)}e^{it_0|\eta|^2}1_{I_M},$$ where $\{\eps_j: j\in J\}$ is a set of
i.i.d. ~random variables taking $\pm1$ with an equal probability $1/2$. By using the Khintchine
inequality, we obtain
\begin{equation*}
\mathbb{E}\left(\|(fd\sigma_1)^{\vee} (gd\sigma_2)^{\vee}\|_{L^1_{t,x}(\R\times A_R)}\right)\sim
\|(\sum_{j}|(f_jd\sigma_1)^{\vee}(gd\sigma_2)^{\vee}|^2)^{1/2}\|_{L^1_{t,x}(\R\times A_R)},
\end{equation*} where $\mathbb{E}(X)$ denotes the expectation of the random variable $X$. By
the heuristic approximation \eqref{prop: heur-four-bess}, the right-hand side of the above is
comparable to
\begin{align*}
R^{\frac{n-2}{2}}\int_{R/2}^{R} \int_{\R} &\left(\sum_{j}\left|\int_{c_j-M^{2}}^{c_j} e^{i[(\pm
r-r_0)s_1-(t-t_0)s_1^2]}ds_1\right.\right. \times \\
&\left.\left.\times \int_{I_M}
e^{-i(t-t_0)s_2^2}(d\mu)^{\vee}(rs_2e_1)ds_2\right|^2\right)^{1/2}dt\,dr,
\end{align*} where $c_j=1+j M^{2}$.
We choose $r\in [R/2, R]$ and $t\in \R$ such that $$R/100\le r-r_0\le R/50,\, M^{-2}/100\le
t-t_0\le M^{-2}/50.$$ This gives $|[(r-r_0)2(t-t_0)](s_1-c_j)|\le c$ and
$-\frac{(r+r_0)+2(t-t_0)}{2(t-t_0)}<-1$. Then by the principle of non-stationary phase and the
triangle inequality, the above is bounded below by $RM^{-2}|J|^{1/2} M^3$, where $|J|\sim M^{-2}$
denotes the cardinality of the set $J$. Hence the left-hand side of \eqref{eq:q=1, 2, 1<R<1/M}
$\gtrsim R^{n/2}$ and the right-hand side of \eqref{eq:q=1, 2, 1<R<1/M} $\lesssim R^{n/2}$ for
$2\le p\le \infty$. Thus the estimates on the line $q=1$ when $2\le R\le 1/M$ are sharp.
\end{example}

In the following example, we will see the estimates in the region $III$ in Figure \ref{fig:
bilin-dya} determined by the lines $q=2$ and $q=4$ are sharp in the case $2\le R\le 1/M$. We will
do computations for the estimates on the line $q=2$ below.

\begin{example}[\textbf{III}]\label{ex: III-med}
If $2\le R\le 1/M$, $R^{1/2}M^{(n-1)/p'}$ is best possible in the following inequality.
\begin{equation}\label{eq:q=2,4, 1<R<1/M}
\|(fd\sigma_1)^{\vee} (gd\sigma_2)^{\vee}\|_{L^2_{t,x}} \lesssim
R^{1/2}M^{(n-1)/p'}\|f\|_{L^p(S_1)}\|g\|_{L^p(S_2)},
\end{equation} where $2\le p\le \infty$.

We take
\begin{align*}
f(|\xi|^2,\xi)&=|\xi|^{-\frac{n-2}{2}}e^{i(-r_0|\xi|+t_0|\xi|^2)}1_{\{1\le |\xi|\le 2\}},\\
g(|\eta|^2,\eta)&=|\eta|^{-(n-2)}e^{it_0|\eta|^2}1_{I_M},
\end{align*} where $r_0\in [R/2, R]$ and $t_0\in \R$. Then by the heuristic approximation
 \eqref{prop: heur-four-bess}, the left-hand side of
\eqref{eq:q=2,4, 1<R<1/M} is comparable to
\begin{align*}
\left(\int_{R/2}^{R}\int_{\R} \left|\int_{I_1}e^{i[(\pm
r-r_0)s_1-(t-t_0)s_1^2]}ds_1\int_{I_M}e^{-i(t-t_0)s_2^2}
(d\mu)^{\vee}(rs_2e_1)ds_2\right|^2dt\,dr\right)^{1/2}.
\end{align*}

We choose $r\in [R/100, R/50]$ and $t\in \R$ such that $\frac {r-r_0}{2(t-t_0)}\in I_1$. Then $r$
and $t$ are in the region of size $\sim R^2$. The principles of stationary phase and
non-stationary phase again give, for any $N\ge 0$,
$$\left|\int_{I_1}e^{i[(r-r_0)s_1-(t-t_0)s_1^2]}ds_1\right|\gtrsim R^{-\frac{1}{2}},\,
\left|\int_{I_1}e^{-i[(r+r_0)s_1+(t-t_0)s_1^2]}ds_1\right|\lesssim_N R^{-N}.$$ With this choice of
$r$ and $t$, we have
$\left|\int_{I_M}e^{-i(t-t_0)s_2^2}(d\mu)^{\vee}{(irs_2\omega)}ds_2\right|\gtrsim M$. Hence from
the triangle inequality, the left-hand side of \eqref{eq:q=2,4, 1<R<1/M} $\gtrsim R^{1/2}M$. On
the other hand, the right-hand side of \eqref{eq:q=2,4, 1<R<1/M} $\lesssim R^{1/2}M$. Thus we see
that the estimates on the line $q=2$ when $2\le R\le 1/M$ are sharp.
\end{example}

The next example will show that the estimates in the region $IV$ in Figure \ref{fig: bilin-dya}
determined by $L^2\times L^2\to L^2$, $L^4\times L^4\to L^4$ and $L^1\times L^1\to L^{\infty}$ are
sharp. We will do the computations for the estimate $L^2\times L^2\to L^2$.

\begin{example}[\textbf{IV}]\label{ex: IV-med}
If $2\le R\le 1/M$, $R^{1/2}M^{(n-1)/2}$ is best possible in the following inequality.
\begin{equation}\label{ex:loc-17}
\|(fd\sigma_1)^{\vee} (gd\sigma_2)^{\vee}\|_{L^2_{t,x}}\lesssim
R^{\frac12}M^{\frac{n-1}{2}}\|f\|_{L^2(S_1)}\|g\|_{L^2(S_2)}.
\end{equation}

We take \begin{align*}
f(|\xi|^2, \xi)&=|\xi|^{-\frac{n-2}{2}}e^{i(-r_0|\xi|+t_0|\xi|^2)}1_{\{1 \le |\xi|\le 1+R^{-1/2}\}},\\
g(|\eta|^2, \eta)&=|\eta|^{-(n-2)}e^{it_0|\eta|^2}1_{I_M},
\end{align*} where $r_0\in [R/2, R]$ and $t_0\in \R$. By the heuristic approximation for
$(fd\sigma)^{\vee}$, the left-hand side of \eqref{ex:loc-17} is comparable to
\begin{align*}
\left(\int_{R/2}^{R}\int_{\R} \left|\int_{1}^{1+R^{-1/2}}e^{i[(\pm
r-r_0)s_1-(t-t_0)s_1^2]}ds_1\int_{I_M}e^{-i(t-t_0)s_2^2}
(d\mu)^{\vee}{(irs_2e_1)}ds_2\right|^2dt\,dr\right)^{1/2}.
\end{align*}

We choose $r\in [R/2, R]$ and $t\in \R$ such that
$$|(r-r_0)-2(t-t_0)|\le R^{1/2}/100,\, R^{1/2}/100\le t-t_0\le R^{1/2}/50.$$
Then $r$ and $t$ are located in the intersection area of two tubes which has size of $R\times
R^{1/2}$. Hence the left-hand side of \eqref{ex:loc-17} $\gtrsim R^{1/4}M$. On the other hand, its
right-hand side $\lesssim R^{1/4}M$. Thus we see that the estimate $L^2\times L^2\to L^2$ when
$2\le R\le 1/M$ is sharp.
\end{example}

The following example will show that the estimates in the region $V$ in Figure \ref{fig:
bilin-dya} determined by the lines $q=4$, $q=\infty$ and $q=3p'$ are sharp.
\begin{example}[\textbf{V}]\label{ex:V-med}
If $2\le R\le 1/M$, $R^{-(n-2)/2}M^{{(n-1)}/{p'}}$ is best possible in the following inequality
\begin{equation}\label{eq: loc-16}
\|(fd\sigma_1)^{\vee}(gd\sigma_2)^{\vee}\|_{L^{\infty}_{t,x}} \lesssim
R^{-\frac{n-2}{2}}M^{\frac{n-1}{p'}}\|f\|_{L^p(S_1)}\|g\|_{L^p(S_2)},
\end{equation} where $1\le p\le \infty$.

We take
\begin{align*}f(|\xi|^2, \xi)&=|\xi|^{-\frac{n-2}{2}}e^{i(-r_0|\xi|+t_0|\xi|^2)}1_{\{1\le|\xi|\le 2\}},\\
g(|\eta|^2, \eta)&=|\eta|^{-(n-2)}e^{it_0|\eta|^2}1_{I_M},
\end{align*} where $r_0\in [R/2, R]$ and $t_0\in \R$. They are chosen such that
$\|(fd\sigma_1)^{\vee}\|_{L^{\infty}_{t,x}}$ and $\|(gd\sigma_2)^{\vee}\|_{L^{\infty}_{t,x}}$ can
be realized at $(t_0,x_0)$ with $|x_0|=r_0$. By the heuristic approximation \eqref{prop:
heur-four-bess},
\begin{align*}
R^{-\frac{n-2}{2}}\left|\int_{I_1}e^{i(\pm
r_0-r_0)s_1}ds_1\int_{I_M}(d\mu)^{\vee}(rs_2e_1)ds_2\right|.
\end{align*} Then from the triangle inequality, the left-hand side of \eqref{eq: loc-16} $\gtrsim
R^{-(n-2)/{2}}M$. On the other hand, its right-hand side $\lesssim R^{-(n-2)/{2}}M$ for $1\le p\le
\infty$. Thus the estimates on the line $q=\infty$ when $2\le R\le 1/M$ are sharp.

When $q=4,4\le p\le \infty$, or $q=3p', 1\le p<4$, the estimates go back to \eqref{eq: loc-15}. We
choose $r\in [R/2, R]$ and $t\in \R$ such that $2\le r-r_0\le 4$ and $2\le t-t_0\le 4$.
\end{example}

\underline{\textit{Case 3}: $R\le 1$.}

In this subcase, we will construct counterexamples to show the estimates \eqref{ex:loc-22} and
\eqref{ex:loc-21} are sharp. We will omit the computations for simplicity.

The following example shows that the estimates in the region $I$ determined by $L^2\times L^2\to
L^1$, $L^2\times L^2\to L^2$ and $L^1\times L^1\to L^{\infty}$ are sharp.
\begin{example}[\textbf{I}]\label{ex:I-sml}
We take
\begin{align*}
f(|\xi|^2,\xi)=F(|\xi|)&=|\xi|^{-(n-2)}e^{it_0|\xi|^2}1_{\{1\le
|\xi|\le 1+M^2\}}, \\
g(|\eta|^2,\eta)=G(|\eta|)&=|\eta|^{-(n-2)}e^{it_0|\eta|^2}1_{I_M},
\end{align*} where $t_0\in \R$. The $r$ and $t$ are chosen such that $\frac R2\le r\le R$ and
$\frac1{100M^2}\le t-t_0\le \frac 1{50M^2}$.
\end{example}

The next example shows that the estimates in the region $II$ determined by the lines $q=1$ and
$q=2$ are sharp.
\begin{example}[\textbf{II}]\label{ex:II-sml}
We define an index set $J:=\{j: 1\le j\le [M^{-2}]\}$. For each $j\in J$, we set
\begin{equation*} f_j(|\xi|^2,\xi)=F_j(|\xi|)=|\xi|^{-(n-2)}e^{it_0|\xi|^2}1_{\{1+(j-1)M^{2}\le
|\xi|\le 1+jM^{2}\}}.
\end{equation*}
Then we define $$f=\sum_{j}\eps_j f_j, \,\,
g(|\eta|^2,\eta)=|\eta|^{-(n-2)}e^{it_0|\eta|^2}1_{I_M},$$ where $\{\eps_j: j\in J\}$ is a set of
i.i.d. ~random variables taking $\pm1$ with an equal probability $1/2$, and the $r$ and $t$ are
chosen such that $R/2\le r\le R$ and $1/2\le t-t_0\le 1$.
\end{example}
The third example shows that the estimate \eqref{ex:loc-21} is sharp. Hence the estimates in the
regions $III$, $IV$ and $V$ when $R\le 1$ are sharp.
\begin{example}[\textbf{III, IV, V}]\label{ex:III,IV,V-sml}
We take
\begin{align*}
f(|\xi|^2,\xi)=F(|\xi|)&=|\xi|^{-(n-2)}e^{it_0|\xi|^2}1_{I_1}, \\
g(|\eta|^2,\eta)=G(|\eta|)&=|\eta|^{-(n-2)}e^{it_0|\eta|^2}1_{I_M},
\end{align*}
where $t_0\in \R$. The $r$ and $t$ will be are chosen such that $ \frac R2\le r\le R$ and
$\frac12\le t-t_0\le 1$.
\end{example}
Thus the proof of Theorem \ref{thm:dyadic-bilin} is complete.

\section{Connection with the restriction estimates for the sphere or the hypersurface of elliptic
type}\label{sec:other-surf} In this section we are concerned with whether the analogous results of
Theorems \ref{thm:dyadic-lin} and \ref{thm:dyadic-bilin} remain valid if $S$ is replaced with the
lower third of the sphere $S^{n-1}$ or a cylindrically symmetric and compact hypersurface of
elliptic type.

Let us first consider the case where the paraboloid is replaced by the sphere $\S^{n-1}$ in
$\R^{n}$. Suppose $f$ is a cylindrically symmetric function supported on a compact set of
$\S^{n-1}$, $S:=\{(-\sqrt{1-|\xi|^2},\xi)\in \R\times\R^{n-1}: M\le |\xi|\le 2M\}$, where $0< M\le
1/6$. Then
\begin{equation}\label{eq:sphere-1}
(fd\mu)^{\vee}(t,x)=\int_{M\le |\xi|\le 2M} e^{i(x\cdot \xi-t\sqrt{1-|\xi|^2})}F(|\xi|)d\xi,
\end{equation}
where $d\mu$ is the surface measure of the sphere and $F(|\xi|)=f(-\sqrt{1-|\xi|^2}, \xi)$.

Since $f$ is cylindrically symmetric, we see that $(fd\mu)^{\vee}$ is also cylindrically
symmetric. Then if we change \eqref{eq:sphere-1} to the polar coordinates to obtain
\begin{equation}\label{eq:sphere-2}
(fd\mu)^{\vee}(t,r)=\int_{I} F(s)e^{-it\sqrt{1-s^2}}(d\mu)^{\vee}(rse_1) s^{n-2}ds,
\end{equation} where $I=[M,2M]$. By the Taylor expansion of $\sqrt{1-s^2}$ at $s=0$,
\begin{equation}
-\sqrt{1-s^2}=-1+\frac 12 s^2+C_1(s)s^4,
\end{equation} where $C_1(s)\sim 1$ for all $s\in I$. Then from \eqref{eq:sphere-2},
\begin{equation}\label{eq:sphere-3}
|(fd\mu)^{\vee}(2t,r)|=\left|\int_{I} F(s)e^{it(s^2+C_2s^4)}(d\mu)^{\vee}(rse_1) s^{n-2}ds\right|,
\end{equation}
where $C_2(s)\sim 1$ for all $s\in I$. The factor ``2" in $(fd\mu)^{\vee}(2t,r)$ is artificial
since we are going to integrate $t$ in $\R$. We make two key observations similar to those we used
in Theorem \ref{thm:dyadic-lin} and \ref{thm:dyadic-bilin} as follows.
\begin{itemize}
\item Since $0\le M\le 1/6$,
\begin{equation*}\frac {d(s^2+C_2(s)s^4)}{ds}\sim s,\,\frac{d^2(s^2+C_2(s)s^4)}{d^2s}\sim 1.
\end{equation*} Heuristically, this condition means that if we change variables $s^2+C_2(s)s^4\to a$,
 it is almost like changing $s^2\to a$. Hence the analogous result to Proposition \ref{prop:lin-q=3p'}
 will hold for the lower third of the sphere.

\item Form the geometric properties of the sphere,
$$\|d\mu_1*d\mu_2\|_{L^{\infty}_{t,r}}\lesssim O(1),$$
when $d\mu_1$ and $d\mu_2$ are the canonical Lebesgue measure of two arcs of size $O(1)$ supported
on the sphere $\S^1$ but separated by a distance $O(1)$. Hence the analogous result to Proposition
\ref{prop:lin-q=4} will hold.
\end{itemize}

Those observations enable us to run all the arguments in Theorem \ref{thm:dyadic-lin} and
\ref{thm:dyadic-bilin}.

We now turn to the case where $S$ is replaced by a cylindrically symmetric and compact
hypersurface $S$ of elliptic type, i.e., $S$ is of the form
\begin{equation}\label{eq:elliptic set}
S:=\{(\tau, \xi)\in \R\times \R^{n-1}: \tau=|\xi|^2+\eps \phi(\xi) \}
\end{equation}
where the error function $\phi(\xi)$ is radial and smooth, and $\eps$ is a sufficiently small
parameter depending on the smooth norms of $\phi$ and on the size of $S$, or more generally on the
separation of $S_1$ and $S_2$. In other words, $S$ is the small perturbation of the standard
paraboloid. By similar observations we made on the sphere, we can establish the analogous results
to Theorems \ref{thm:dyadic-lin} and \ref{thm:dyadic-bilin} for the cylindrically symmetric
functions compactly supported on $S$ defined in \eqref{eq:elliptic set}.

\section{Connection with Strichartz inequalities of the Schr\"odinger equation}\label{sec:PDE appli}
\subsection{Linear Strichartz estimates}The restriction problem is closely related to that of
estimating solutions to linear PDE such as the wave equation and the Schr\"odinger equation.
Strichartz first observed this connection in \cite{Strichartz:1977}, which initiated the intensive
study on various Strichartz estimates. In this section we will interpret our restriction estimates
regarding $(fd\sigma)^{\vee}$ in terms of the solutions to the Schr\"odinger equations.

Suppose $f(\tau, \xi)$ is a function supported on the paraboloid $S$ in $\R\times\R^{n-1}$.
Functions of the form $u(t,x):=(fd\sigma)^{\vee}$, where $d\sigma$ is the canonical Lebesgue
measure on $S$, can be easily seen to solve the free Schr\"odinger equation
\begin{equation}\label{eq: schrodinger equation}
iu_t+\triangle u=0,\, u(0,x)=u_0(x),
\end{equation}where the spatial Fourier transform $\hat{u}_0(\xi)=f(|\xi|^2, \xi)$.  It is easy to
deduce that $f$ is cylindrically symmetric on $\R\times \R^{n-1}$ if $u_0$ is radial on
$\R^{n-1}$. By this interpretation, the linear estimate $L^p\to L^q$ or the bilinear restriction
estimate $L^p\times L^p\to L^q$ will correspond to certain Strichartz estimates. For instance, the
Tomas-Stein restriction estimate $L^2\to L^{2(n+1)/(n-1)}$ implies the Strichartz estimate
$$\|e^{it\triangle}u_0\|_{L^{\frac{2(n+1)}{n-1}}_{t, x}(\R\times \R^{n-1})}\lesssim
\|u_0\|_{L^2(\R^{n-1})},$$ where we have denoted $u$ by
$e^{it\triangle}u_0(x)=\int_{\R^{n-1}}e^{i(x\xi+t|\xi|^2)}\hat{u}_0(\xi)d\xi$. This is known to be
best possible simply by the scaling property associated to the Schr\"odinger equation. In fact, we
have the following optimal result called the linear Strichartz estimates
\cite{Keel-Tao:1998:endpoint-strichartz},
\begin{equation}\label{eq:lin-stri}
\|e^{it\triangle}u_0\|_{L^q_t L^r_x(\R\times \R^{n-1})}\lesssim \|u_0\|_{L^2(\R^{n-1})}
\end{equation}
if and only if
\begin{equation}\label{eq:lin-restr-range}
\frac{2}{q}+\frac{n-1}{r}=\frac{n-1}{2}, \, q\ge 2,\, r\ge 2, \, (q,r,n)\neq (2,\infty, 3).
\end{equation}

A natural question arises: if we assume that $\hat{u}_0$ is radial and supported on a compact set
$U:=\{\xi\in \R^{n-1}: M\le |\xi|\le 2M\}$ with dyadic $M>0$, do we have further estimates
available? The answer is confirmed in Corollary \ref{cor:lin-restr}. In particular, we have
\begin{corollary} Suppose $u_0$ is defined as above. Then for any $q>\frac{4n-2}{2n-3}$,
\begin{equation}\label{ex:loc-23}
\|e^{it\triangle}u_0\|_{L^{q}_{t, x}(\R\times \R^{n-1})}\lesssim
M^{\frac{n-1}{2}-\frac{n+1}{q}}\|u_0\|_{L^2(\R^{n-1})}.
\end{equation}
\end{corollary}

\begin{remark}
For such functions, one can easily extend the current range \eqref{eq:lin-restr-range} for the
linear mixed norm Strichartz estimates \eqref{eq:lin-stri} by interpolating them with the estimate
\eqref{ex:loc-23}.
\end{remark}

In another direction, one can also obtain various weighted Strichartz estimates. This type of
estimates for radial data has proven very useful in establishing the global well-posedness and
scattering results for certain Schr\"odinger equations, see e.g.,
\cite{Tao-Visan-Zhang:2007:radial-NLS-higher}. In \cite{Vilela:2001:radial-schrod}, Vilela showed
that, assuming $u_0\in L^2(\R^{n-1})$ to be radial,
\begin{equation}\label{eq:vilela}
\|D^s_x e^{it\triangle}u_0\|_{L^2_{t,x}(|x|^{-\alpha})}\lesssim \|u_0\|_{L^2},
\end{equation} if and only if $\alpha=2(1-s)$, $1<\alpha<n-1$ and $n\ge
3$, where $D^s_x f$ is defined via the spatial Fourier transform by $\widehat{D^s_x
f}(\xi)=|\xi|^s\hat{f}(\xi)$. The ``only if'' part is given in \cite{Vilela:2001:radial-schrod} by
the decay estimate of $(d\sigma)^{\vee}$ and scaling. Here we will give another proof of the
``if'' part by using the linear dyadic restriction restriction estimates given by Theorem
\ref{thm:dyadic-lin}.
\begin{proof}
We first assume that $u_0$ has dyadically localized frequency, i.e.,  $\hat{u}_0$ supported on the
set $\{\xi: M/2\le |\xi|\le M\}$ with dyadic $M$. Then we set $f(|\xi|^2, \xi)=\hat u_0(M\xi)$,
i.e., $f\in \L_1$. Then from the estimate $L^2\to L^2$ in Theorem \ref{thm:dyadic-lin}, we obtain,
$\forall \, \eps>0$,
$$\||x|^{-(1+\eps)/2}(fd\sigma)^{\vee}\|_{L^2{(\R\times \{|x|\ge 1\})}}\lesssim_{\eps}
\|f\|_{L^2(\R^{n-1})}.$$ If we restrict $0<\eps<n-2$, by the Plancherel theorem in $t$,
$$\||x|^{-(1+\eps)/2}(fd\sigma)^{\vee}\|_{L^2{(\R\times \{|x|\le 1\})}}
\lesssim_{\eps} \|f\|_{L^2(\R^{n-1})}.$$ Hence
$$\||x|^{-(1+\eps)/2}(fd\sigma)^{\vee}\|_{L^2{(\R\times \R^{n-1})}}\lesssim_{\eps}
\|f\|_{L^2(\R^{n-1})}.$$ By re-scaling by $M$,
$$\||x|^{-(1+\eps)/2}M^{(1-\eps)/2}e^{it\triangle}u_0\|_{L^2{(\R\times \R^{n-1})}}\lesssim_{\eps}
\|u_0\|_{L^2(\R^{n-1})}.$$ By the weighted H\"ormander-Mikhlin theorem \cite[Lemma
2.2]{Tao-Visan-Zhang:2007:radial-NLS-higher},
$$\|D^{(1-\eps)/2}e^{it\triangle} u_0\|_{L^2(|x|^{-(1+\eps)})}\lesssim_{\eps} \|u_0\|_{L^2(\R^{n-1})}.$$
Setting $s=(1-\eps)/2$ and $\alpha=1+\eps$, we obtain \eqref{eq:vilela} for frequency localized
$u_0$. Then we follow the approach of using the Khintchine inequality to prove the
Littlewood-Paley inequality and use the weighted inequalities for singular integrals\cite[Chapter
5, Corrollary 4.2]{Stein:1993} ($|x|^{-\alpha}$ is a $A_2$ weight) to obtain \eqref{eq:vilela}.
\end{proof}

\subsection{Bilinear Strichartz estimates}
Form the linear strichartz estimates \eqref{eq:lin-stri}, we see their bilinear analogues,
\begin{equation}\label{eq:bilin-stri}
\|e^{it\triangle}u_0e^{it\triangle}v_0\|_{L^q_t L^r_x(\R\times \R^{n-1})}\lesssim
\|u_0\|_{L^2(\R^{n-1})}\|v_0\|_{L^2(\R^{n-1})}
\end{equation} if and only if
\begin{equation}\label{eq:bilin-stri-r}
 \frac{2}{q}+\frac{n-1}{r}=n-1;\, q, r\ge 1;\, (q,r,n)\neq(1,\infty, 3).
\end{equation} For the necessity of excluding the endpoint $(1,\infty, 3)$, see
\cite{Tao:2006:counterexample-bilinear-strichartz}.

The estimate \eqref{eq:bilin-stri} becomes more interesting when we assume $u_0$ and $v_0$ are
compactly supported and separated by a distance comparable to $O(1)$. In this case, we expect that
there are more estimates available. For instance, when $q=r$, Klainerman and Machedon
\cite{Klainerman-machedon:1993:spa-time-null-form} conjectured that \eqref{eq:bilin-stri} holds if
and only if $q=r\ge(n+2)/{n}$. The exponent $(n+2)/n$ is best possible, see e.g.,
\cite{Tao-Vargas-Vega:1998:bilinear-restri-kakeya}, \cite{Tao:2003:paraboloid-restri}. This
conjecture has been verified by Tao in \cite{Tao:2003:paraboloid-restri} up to the endpoint
$(n+2)/n$. The analogous results in the cone setting were established by Wolff in the non-endpoint
case \cite{Wolff:2001:restric-cone} and Tao in the endpoint case \cite{Tao:2001:endpoint-cone}.

As shown in Corollary \ref{cor:bilin-restr}, we have further estimates available if we assume that
$\widehat{u}_0$ and $\widehat{v}_0$ are radial functions and compactly supported on $U_1=\{\xi\in
\R^{n-1}: M_1/2\le |\xi|\le M_1\}$ and $U_2=\{(\xi\in \R^{n-1}: M_2/2\le |\xi|\le M_2\}$,
respectively. Here $M_1>0$, $M_2>0$ are dyadic numbers satisfying $M_2\le M_1/4$. For instance, as
a corollary of Theorem \ref{thm:dyadic-bilin}, we have the following bilinear Strichartz estimates
by interpolation and summing in dyadic $R$.
\begin{corollary}
Suppose $u_0, v_0$ are defined as above. Then
\begin{itemize}
\item for $\frac{n}{n-1}< q\le 2$,
$$\| e^{it\triangle}u_0e^{it\triangle}v_0\|_{L^q_{t,x}(\R\times\R^{n-1})}\lesssim
M_1^{-\frac12}M_2^{\frac{2n-1}{2}-\frac{n+1}{q}}\|u_0\|_{L^2(\R^{n-1})}\|v_0\|_{L^2(\R^{n-1})}.$$

\item for $2\le q\le \frac{2(2n-1)}{2n-3}$,
$$\| e^{it\triangle}u_0e^{it\triangle}v_0\|_{L^q_{t,x}(\R\times\R^{n-1})}\lesssim
M_1^{-\frac{3}{2q}+\frac{1}{4}}M_2^{\frac{4n-5}{4}-\frac{2n-1}{2q}}\|u_0\|_{L^2(\R^{n-1})}
\|v_0\|_{L^2(\R^{n-1})}.$$

\item for $q\ge \frac{2(2n-1)}{2n-3}$,
$$\|e^{it\triangle}u_0e^{it\triangle}v_0\|_{L^q_{t,x}(\R\times\R^{n-1})}\lesssim
M_1^{\frac{n-1}{2}-\frac{n+1}{q}}M_2^{\frac{n-1}{2}}\|u_0\|_{L^2(\R^{n-1})}\|v_0\|_{L^2(\R^{n-1})}.$$
\end{itemize}
\end{corollary}

\begin{remark}
It is clear that, $\forall \, q>\frac{n}{n-1}$, $n\ge 3$ and $M_1, M_2\sim 1$,
$$\| e^{it\triangle}u_0e^{it\triangle}v_0\|_{L^q_{t,x}(\R\times\R^{n-1})}\lesssim
\|u_0\|_{L^2(\R^{n-1})}\|v_0\|_{L^2(\R^{n-1})},$$ which improves $q>\frac{n+2}{n}$.
\end{remark}

\begin{remark}
When $q=2$ and $n\ge 3$, we have the following sharp estimates for $u_0, v_0$ defined as above,
\begin{equation*}
\|e^{it\triangle}u_0e^{it\triangle}v_0\|_{L^2_{t,x}(\R\times\R^{n-1})}\lesssim
M_1^{-\frac12}M_2^{\frac{n-2}{2}}\|u_0\|_{L^2(\R^{n-1})}\|v_0\|_{L^2(\R^{n-1})},
\end{equation*} which generalizes Bourgain's following estimates to all dimensions
\begin{align*}
\|e^{it\triangle}u_0e^{it\triangle}v_0\|_{L^2_{t,x}(\R\times\R^2)}\lesssim
M_1^{-\frac12}M_2^{\frac 12}\|u_0\|_{L^2(\R^{2})}\|v_0\|_{L^2(\R^{2})},\\
\|e^{it\triangle}u_0e^{it\triangle}v_0\|_{L^2_{t,x}(\R\times\R^3)}\lesssim
M_1^{-\frac12}M_2\|u_0\|_{L^2(\R^{3})}\|v_0\|_{L^2(\R^{3})}.\end{align*} But we remark that
Bourgain's estimates are for general $u_0$ and $v_0$ without the radial assumption, see
\cite{Bourgain:1993:Schrodinger-lattice}, \cite{Bourgain:1999:colloqium-book}.
\end{remark}

\bibliography{refs}
\bibliographystyle{plain}
\end{document}